\documentclass{imanum}
\usepackage{graphicx}

\jno{}
\received{\today}
\usepackage{tristan}

\usepackage{a4wide}
\usepackage{subfigure}

\newcommand{\hkm}[1]{\sobh{#1}_m}
\newcommand{\hatu}{\widehat{u}}
\newcommand{\hatv}{\widehat{v}}
\renewcommand{\enorm}[1]{\ensuremath{\Norm{#1}}_{dG}}
\renewcommand{\bih}[2]{\ensuremath{\cA_h\qp{#1,#2}}}

\numberwithin{equation}{section}

\begin{document}

\title{Reduced relative entropy techniques for a posteriori analysis of multiphase problems in elastodynamics}
\shorttitle{RRE a posteriori indicators in elastodynamics}

\author{%
  {\sc
    Jan Giesselmann
    \thanks{\tt{jan.giesselmann@mathematik.uni-stuttgart.de}}
    \\
    [2pt]{
      Institute of Applied Analysis and
      Numerical Simulation, University of Stuttgart,
      Pfaffenwaldring 57, D-70563 Stuttgart, Germany
      and}
  }
  \\
  {\sc 
    Tristan Pryer
    \thanks{Corresponding author. {\tt{T.Pryer@reading.ac.uk}}}
    \\
    [2pt]{
      Department of Mathematics and
      Statistics, Whiteknights, PO Box 220,
      Reading, GB-RG6 6AX, England UK. }
  }
}
\shortauthorlist{J. Giesselmann and T. Pryer}

\maketitle

\begin{abstract}
  {We give an a posteriori analysis of a semidiscrete discontinuous
  Galerkin scheme approximating solutions to a model of multiphase
  elastodynamics, which involves an energy density depending not only
  on the strain but also the strain gradient.  A key component in the
  analysis is the \emph{reduced relative entropy} stability framework
  developed in [Giesselmann 2014]. This framework allows energy type
  arguments to be applied to continuous functions. Since we advocate
  the use of discontinuous Galerkin methods we make use of two
  families of reconstructions, one set of discrete reconstructions
  [Makridakis and Nochetto 2006] and a set of elliptic reconstructions
  [Makridakis and Nochetto 2003] to apply the reduced relative entropy
  framework in this setting.}
  {discontinuous Galerkin finite element method, a posteriori error
  analysis, multiphase elastodynamics, relative entropy, reduced
  relative entropy.}
\end{abstract}

\section{Introduction}

Our goal in this work is to introduce the \emph{reduced relative
  entropy} technique as a methodology for deriving a posteriori
estimates to finite element approximations of a problem arising in
elastodynamics. In particular, this work is concerned with providing a
rigorous a posteriori error estimate for a semi (spatially) discrete
discontinuous Galerkin scheme approximating a model for shearing
motions of an elastic bar undergoing phase transitions between phases
which correspond to different (intervals of) shears, \eg austenite and
martensite.  In this model the energy density depends not only on the
strain but also on the strain gradient. Such models are often referred
to as models of ``first strain gradient'' or ``second gradient'' type
\cite{JTB02,JLCD01}.  The latter is due to the fact that the strain
gradient is the second gradient of the deformation.

The relative entropy technique is the natural stability framework for
problems in nonlinear elasticity. Introduced, for hyperbolic conservation laws,
in \cite{Daf79,Dip79}, this technique is based on the fact that
systems of conservation laws are usually endowed with an
entropy/entropy flux pair. For conservation laws describing physical
phemonena this notion of entropy follows from the physical one. The
entropy/entropy flux pair also gives rise to an admissibility
condition for weak solutions which leads to the notion of entropy
solutions. It can also be used to define the notion of relative
entropy between two solutions. In case of a convex entropy the
relative entropy is equivalent to the square of the $\leb{2}$
distance. In hyperbolic balance laws and related problems stability
estimates based on the relative entropy framework are by now standard
if the entropy is at least quasi or polyconvex, see \cite{Daf10} and
references therein.

The model we consider in this work does not fall into this framework
however.  It describes a multiphase process and, therefore, the
energy density is expected to have a multiwell shape and, in
particular, is neither quasi nor polyconvex. Indeed, the first order
part of the model is no longer hyperbolic but of hyperbolic/elliptic
type.  It is well known that in this situation entropy solutions (to
the first order problem) are not unique \cite{LeF02} and either
kinetic relations have to be introduced or regularisations need to be
considered.  We follow the second approach and consider a model
including a second gradient/capillarity regularisation which also
allows for viscosity.

To account for the non-convexity of the energy, we will employ the
\emph{reduced relative entropy technique} which is a modification of
the classical arguments used in the relative entropy framework in
which we only consider the \emph{convex contributions} of the entropy
\cite{Gie}. Roughly speaking it uses the higher order regularizing
terms in order to compensate for the non-convexity of the energy. The
reduced relative entropy technique is only applicable when studying
continuous solutions to the problem, as such, is not immediately
applicable to discontinuous Galerkin approximations. Our methodology
consists of applying appropriate \emph{reconstructions} of the
discrete solution into the continuous setting, then using the reduced
relative entropy technique to bound the difference of the
reconstruction and the exact solution. 

The numerical analysis of schemes approximating regularized
hyperbolic/elliptic problems, like the model at hand or the
Navier-Stokes-Korteweg system in compressible multiphase flows, is
rather limited
\cite{CL01,Die07,BP13,GiesselmannMakridakisPryer:2014,JTB02,Gie_14b,JLCD01},
and the available works mainly focus on the stability of schemes.
Previous works on discontinuous Galerkin methods for scalar dispersive
equations can be found in \cite{CS08, BCKX11, XS11}. See also \cite{OrtnerSuli:2007} for discontinuous Galerkin approximating hyperbolic nonlinear elastodynamics in several space dimensions. Note that the results of \cite{OrtnerSuli:2007} do not require convexity of the energy density but rely on a weaker G\"arding type inequality which is in agreement with constitutive laws of real materials without phase transitions.

A benefit of our approach is that we are able to derive both a priori,
assuming sufficient regularity on the solution
\cite{GiesselmannPryer:2014}, and a posteriori error estimates based on
similar techniques. In the first instance, we apply this methodology
to a regularisation of the equations of nonlinear elastodynamics
including both viscous and dispersive regularising terms. In the case that dispersion regularisation is
small, solutions to the equations display thin layers which are
physically interpreted as phase boundaries. 

In this work, for clarity, we study the one dimensional setting.  Our
analysis is fully extendable to the multidimensional setting
discussed in the second part of \cite{Gie}, assuming an appropriate
discrete reconstruction operator can be constructed (see Remark
\ref{rem:multid}). We make the
important observation that the a posteriori error bounds we derive are applicable
as the viscous parameter tends to zero but blow up when the dispersion
parameter tends to zero. We also expect that our results can be extended to a wider class of problems, for example, 
the (multidimensional) Navier-Stokes-Korteweg equations, although in
that case certain technical restrictions will be necessary; \eg all
involved densities need to be bounded away from vacuum.

The rest of the paper is organised as follows: In \S \ref{sec:model}
we introduce the model problem together with some of its properties
and formalise our notation. In \S \ref{sec:rre} we give a summary of
the reduced relative entropy technique which we use to prove a
stability result in Theorem \ref{the:rreb}. In \S
\ref{sec:discretisation} we state the discretisation of the model
problem, some of its properties and introduce the operators which we
require for the a posteriori analysis. In \S \ref{sec:a posteriori} we
state our main result, which is a computable a posteriori indicator
for the error in the natural entropy norm. Finally, in
\S \ref{sec:numerics} we give summarise extensive numerical results.

\section{Model description and properties}
\label{sec:model}

The specific class of problem which we consider here models the
shearing motion of an elastic bar undergoing phase transitions between
say austenite and martensite phases \cite{AK91}. These models are
based on the isothermal nonlinear equations of elastodynamics.  In one
spatial dimension, they are 
\begin{equation}
  \label{eq:hyp-ell-model}
  \begin{split}
    \pd t u - \pd x v &= 0
    \\
    \pd t v - \pd x W'(u) &= 0,
  \end{split}
\end{equation}
where $u$ is the strain, $v$, the velocity and $W=W(u)$ is the energy
density, which is given by a constitutive relation.
Notice that this may also be rewritten as a nonlinear wave equation
\begin{equation}
  \pd {tt} y - \pd x W'(\pd x y) = 0,
\end{equation}
for the displacement field $y$ which satisfies $\pd x y = u$.  If
\eqref{eq:hyp-ell-model} describes a multiphase situation $W$ has a
multiwell shape and, in particular, is not convex.  This makes
\eqref{eq:hyp-ell-model} a problem of mixed hyperbolic/elliptic type.
For such problems entropy solutions, which are standard in the study
of hyperbolic conservation laws, are not unique.  There are two
methods in order to regain uniqueness of solutions: Either a kinetic
relation, singling out the correct phase transitions,
\cite[c.f.]{AK91} can be imposed or the problem can be regularized
\cite[c.f.]{Sle83,Sle84}.

In this work we focus on the problem
\begin{equation}
  \label{eq:model-prob}
  \begin{split}
    \pdt u - \pd x v &= 0
    \\
    \pdt v - \pd x W'(u) &= \mu \pd {xx} v - \gamma \pd {xxx} u
    \\
    u(x,0) &= u_0(x) \\
    v(x,0) &= v_0(x),
  \end{split}
\end{equation}
where $\mu \geq 0 \AND \gamma > 0$ denote the strength of viscous and
capillarity effects. We will not make any precise assumptions on the
convex and concave parts of $W$ but simply assume $W \in
\cont{3}(\rR,[0,\infty)),$ allowing for all kinds of (regular)
multiwell shapes.

\begin{remark}[State space]
  We could also apply our theory in case $W$ is only defined on some
  open interval $I \subset \rR,$ as would be the case if
  \eqref{eq:model-prob} were to describe compressible fluid flows in a
  pipe or longitudinal motions of an elastic bar.  However, in that
  case we would have to impose the condition that the solutions only
  take values inside a convex and compact subset of the interval $I$, however, for clarity of exposition we will not consider this scenario here.
\end{remark}

We couple (\ref{eq:model-prob}) with periodic
boundary conditions. With that in mind we will denote $S^1$ to be the
\emph{one sphere}, \ie the unit interval with coinciding end points.
Again, note that under sufficient regularity assumptions
(\ref{eq:model-prob}) is equivalent to the wave like equation
\begin{equation}
  \pd {tt} y  - \pd x W'(\pd x y) = \mu \pd {xxt} y - \gamma \pd{xxxx} y.
\end{equation}

We will use standard notation for Sobolev spaces
\cite{Cia78,Evans:1998}
\begin{equation}
  \sobh{k}(S^1) 
  := 
  \ensemble{\phi\in\leb{2}(S^1)}
  {\D^{\alpha}\phi\in\leb{2}(S^1), \text{ for } \alpha\leq k},
\end{equation}
which are equipped with norms and semi-norms
\begin{gather}
  \Norm{u}_{k}^2
  := 
  \Norm{u}_{\sobh{k}(S^1)}^2 
  = 
  \sum_{{\alpha}\leq k}\Norm{\D^{\alpha} u}_{\leb{2}(S^1)}^2 
  \\
  \AND \norm{u}_{k}^2 
  :=
  \norm{u}_{\sobh{k}(S^1)}^2 
  =
  \sum_{{\alpha} = k}\Norm{\D^{\alpha} u}_{\leb{2}(S^1)}^2
\end{gather}
respectively, where derivatives $\D^{\alpha}$ are understood in a
weak sense. In addition, let
\begin{equation}
  \hkm{k}(S^1)
  := 
  \ensemble{\phi\in{\sobh{k}(S^1)}}
  {\int_{S^1} \phi = 0}.
\end{equation}
We also make use of the following notation for time dependent Sobolev
(Bochner) spaces:
\begin{equation}
  \cont{i}(0,T; \sobh{k}(S^1))
  :=
  \ensemble{u : [0,T] \to \sobh{k}(S^1)}
           {u \text{ and $i$ temporal derivatives are continuous}}.
\end{equation}

\begin{theorem}[Existence of strong solutions {\cite[Cor 2.4]{Gie}}]
  \label{the:existence}
  Let $u_0\in\hkm{3}(S^1)$, $v_0\in\hkm{2}(S^1)$ and $T,\mu,\gamma>0$, then
  (\ref{eq:model-prob}) admits a unique strong solution
  \begin{equation}
    \qp{u,v}
    \in
    \cont{0}
    \qp{0,T;\hkm{3}(S^1)}
    \cap 
    \cont{1}
    \qp{0,T;\hkm{1}(S^1)}
    \times
    \cont{0}
    \qp{0,T;\hkm{2}(S^1)}
    \cap
    \cont{1}
    \qp{0,T;\hkm{0}(S^1)}.
  \end{equation}
\end{theorem}

\begin{remark}[Viscosity]
  For the semi-group techniques employed in the proof of Theorem
  \ref{the:existence} it is required that $\mu>0$.  In contrast, all
  our subsequent estimates also hold in case $\mu=0$ provided
  sufficiently regular solutions exist.
\end{remark}

\begin{lemma}[Energy balance]\label{lem:energy}
  Let $\qp{u,v}$ be a strong solution of (\ref{eq:model-prob}),
  $T,\gamma > 0$ and $\mu \geq 0$ then
  \begin{equation}
    \label{eq:energy}
    \ddt \qp{\int_{S^1} W(u) + \frac{\gamma}{2} \norm{\pd x u}^2 + \frac{1}{2}\norm{v}^2} = - \int_{S^1} \mu \norm{\pd x v}^2.
  \end{equation}
\end{lemma}
\begin{proof}
  Testing the first equation of (\ref{eq:model-prob}) with $W'(u) -
  \gamma \pd{xx} u$ and the second equation of (\ref{eq:model-prob})
  with $v$ and taking the sum, we see
  \begin{equation}
    \begin{split}
      0
      &=
      \int_{S^1}
      \pd t u W'(u) 
      -
      \gamma \pd{xx} u \pdt u 
      -
      W'(u) \pd x v 
      +
      \gamma \pd x v \pd {xx} u
      \\
      &\qquad \qquad \qquad 
      +
      v\pdt v
      -
      v\pd x W'(u)
      -
      \mu v \pd {xx} v
      +
      \gamma v \pd{xxx} u.
    \end{split}
  \end{equation}
  Upon integrating by parts we have 
  \begin{equation}
    0 
    =
    \int_{S^1}
    \pdt u W'(u) 
    +
    \gamma \pd x u \pd{tx} u
    +
    v\pdt v
    +
    \mu \qp{\pd x v}^2,
  \end{equation}
  which yields the desired result.
\end{proof}

\begin{remark}[Strain gradient dependent energy]
  Note that the energy density, \ie the integrand in the left hand
  side of \eqref{eq:energy}, consists of three terms.  The kinetic
  energy $\frac{1}{2} v^2 $ and the potential energy (density) which
  is decomposed additively into a strain dependent nonlinear part $W$
  and a part depending on the strain gradient. This latter term is the
  reason why this type of model is called ``first strain gradient'' or
  ``second (deformation) gradient'' model.
\end{remark}

\begin{remark}[$\leb{\infty}$ bound for $u$]
  Lemma \ref{lem:energy} and the fact that the mean value of $u$ does
  not change in time imply that
  $\Norm{u}_{\leb{\infty}(0,\infty;\sobh{1}(S^1))}$ is bounded in
  terms of the initial data.  As $\sobh{1}(S^1) \subset
  \leb{\infty}(S^1)$ we may immediately infer that
  $\Norm{u}_{\leb{\infty}(S^1 \times (0,\infty)))}$ is bounded in terms
  of the initial data.
\end{remark}

\section{Reduced relative entropy}
\label{sec:rre}

In this section we briefly introduce the reduced relative entropy
technique. Using this we prove the natural stability bounds for the
problem.



\begin{lemma}[Gronwall inequality]
  \label{lem:gronwall}
  {Given $T>0$, let $\phi\in\cont{0}([0,T])$ and $a, b \in
  \leb{1}([0,T])$ all be nonnegative functions with $b$ nondecreasing and satisfying
  \begin{equation}
 \phi(t) \leq \int_0^t a(s) \phi(s) \d s + b(t).
  \end{equation}}
  Then 
  \begin{equation}
    \phi(t) \leq b(t) \exp\qp{\int_0^t a(s) \d s} \Foreach t \in [0, T].
  \end{equation}
\end{lemma}

\begin{definition}[Reduced relative entropy]
  The \emph{reduced relative entropy technique} is a reduction of the
  classical relative entropy technique in the sense that it only
  accounts for the convex part of the entropy. For given $v, \hatv \in \cont{0}([0,T],\leb{2}(S^1))$ and $u, \hatu \in \cont{0}([0,T],\sobh{1}(S^1))$ we define
  \begin{equation}
    \label{eq:rre}
    \eta_R(t) 
    :=
    \frac{1}{2}\int_{S^1} \qp{v(\cdot,t) - \hatv(\cdot,t)}^2 
    +
    \gamma\qp{\pd x u(\cdot,t) - \pd x \hatu(\cdot,t)}^2
    +
    \frac{\mu}{4} \int_0^t \norm{v(\cdot,s)-\hatv(\cdot,s)}_{\sobh{1}(S^1)}^2 \d s
    .
  \end{equation}
\end{definition}

\begin{theorem}[Reduced relative entropy bound]
  \label{the:rreb}
  Let $\qp{u,v}$ be a strong solution to (\ref{eq:model-prob}) and suppose
  $\qp{\hatu, \hatv}$ is a strong solution to the perturbed problem
  \begin{equation}
    \label{eq:perturbed-prob}
    \begin{split}
      \pdt \hatu - \pd x \hatv 
      &=
      0
      \\
      \pdt \hatv - \pd x W'(\hatu)
      &=
      \mu \pd {xx} \hatv - \gamma \pd{xxx} \hatu + \mathfrak{R}
    \end{split}
  \end{equation}
  where $\mathfrak{R}$ is some residual and $\gamma > 0, \mu \geq
  0$. Assume that $\hatu(\cdot,0)=u(\cdot,0)$,
  $\hatv(\cdot,0)=v(\cdot,0)$ and that
  \begin{equation}
    \label{eq:overlineM}
    \overline{M} 
    := 
    \max\qp{
      \Norm{u}_{\leb{\infty}(S^1 \times (0,\infty))}
      ,
      \Norm{\hatu}_{\leb{\infty}(S^1 \times (0,\infty))}
    }
    <
    \infty.
  \end{equation}
  Then, the reduced relative entropy between $\qp{u,v}$ and
  $\qp{\hatu, \hatv}$ satisfies
  \begin{equation}
    \label{eq:rre-bound}
    \eta_R(t) 
    \leq
    \qp{
      \eta_R(0)  
      +
      \Norm{\mathfrak{R}}_{\leb{2}(S^1\times (0,t))}^2
    }
    \exp\qp{\int_0^t K[\hatu](s)\d s}  \Foreach t,
  \end{equation}
  where
  \begin{equation}
    K[\hatu](t) 
    :=
    \max
    \qp{
      \frac{2C_P^2\overline W^2}{\gamma}\Norm{\pd x \hatu(\cdot,t)}_{\leb{\infty}(S^1)}^2
      +
      \frac{2\overline W^2}{\gamma}
      ,
      \frac{3}{2}
    }
    \quad \text{ and }
    \overline{W}= \Norm{W}_{\cont{3}[-\overline M,\overline M]},
  \end{equation}
  where $C_P$ is a Poincar\'e constant.
\end{theorem}
\begin{proof}
  Explicitly computing the time derivative of $\eta_R$ yields
  \begin{equation}
    \dt \eta_R(t)
    =
    \int_{S^1}
    \qp{v - \hatv} \pdt{\qp{v - \hatv}}
    +
    \gamma
    \pd x {\qp{u - \hatu}} \pd {xt} {\qp{u - \hatu}}
    +
    \frac{\mu}{4} \qp{ \partial_x \qp{ v - \hatv}}^2
    .
  \end{equation}
  Using the problem (\ref{eq:model-prob}) and the perturbed problem
  (\ref{eq:perturbed-prob}) we see that 
  \begin{equation}
    \begin{split}
    \dt \eta_R(t)
    &=
    \int_{S^1}
    \qp{v - \hatv} \qp{\pd x W'(u) - \pd x W'(\hatu) - \gamma \pd {xxx} u + \gamma \pd {xxx} \hatu + \mu \pd {xx} v - \mu \pd {xx} \hatv - \mathfrak{R}}
    \\
    &\qquad +
    \gamma
    \pd x {\qp{u - \hatu}} \qp{\pd{xx} v - \pd{xx} \hatv} 
    +  
    \frac{\mu}{4} \qp{ \partial_x  v - \partial_x\hatv}^2.
    \end{split}
  \end{equation}
  Cancellation occurs upon integrating by parts and we have
  \begin{equation}
    \label{eq:rre-1-gamma}
    \begin{split}
      \dt \eta_R(t)
      &=
      \int_{S^1}
      \qp{v - \hatv} \qp{\pd x W'(u) - \pd x W'(\hatu) - \mathfrak{R}}
      -
      \frac{3}{4}\mu \qp{\pd {x} \hatv - \pd {x} v}^2
      \\
      &\leq
      \int_{S^1}
      \qp{v - \hatv} \qp{\pd x W'(u) - \pd x W'(\hatu) - \mathfrak{R}}.
    \end{split}
  \end{equation}
  Making use of Young's inequality we have
  \begin{equation}
    \begin{split}
      \dt \eta_R(t)
      &\leq
      \overline{W}^2 \Norm{\pd x \hatu}^2_{\leb{\infty}(S^1)}\Norm{u - \hatu}^2_{\leb{2}(S^1)}
      +
      \overline{W}^2 \Norm{\pd x u - \pd x \hatu}^2_{\leb{2}(S^1)}
      \\
      &\qquad +
      \Norm{\mathfrak{R}}^2_{\leb{2}(S^1)}
      +
      \frac{3}{4} \Norm{v - \hatv}^2_{\leb{2}(S^1)}.
   \end{split}
  \end{equation}
  Invoking a Poincar\'e inequality yields
  \begin{equation}
    \begin{split}
      \dt \eta_R(t)
      &\leq
      \overline{W}^2 \qp{C_P^2 \Norm{\pd x \hatu}^2_{\leb{\infty}(S^1)} + 1}
      \Norm{\pd x u - \pd x \hatu}^2_{\leb{2}(S^1)}
      +
      \Norm{\mathfrak{R}}^2_{\leb{2}(S^1)}
      +
      \frac{3}{4} \Norm{v - \hatv}^2_{\leb{2}(S^1)}
      \\
      &\leq 
      \max\qp{\overline{W}^2 \qp{C_P^2 \Norm{\pd x \hatu}^2_{\leb{\infty}(S^1)} + 1}\frac{2}{\gamma}, \frac{3}{2}} \eta_R(t)
      +
      \Norm{\mathfrak{R}}^2_{\leb{2}(S^1)}.
    \end{split}
  \end{equation}
  The conclusion follows by invoking the Gronwall inequality given in
  Lemma \ref{lem:gronwall}.
\end{proof}

\begin{corollary}[Uniqueness of solution]
  Under the conditions of Theorem \ref{the:rreb} we have that if
  $\qp{\hatu, \hatv}$ solves (\ref{eq:model-prob}) with no
  residual term we may infer uniqueness of solution. 
\end{corollary}

\begin{remark}[Exponential dependence on problem data]
  Note that the entropy bound in Theorem \ref{the:rreb} depends
  exponentially (in time) on the Lipschitz constant of the perturbed
  solution. This is the main motivation for using reconstructions of
  the discontinuous Galerkin approximations of (\ref{eq:model-prob}).

  In addition the bound depends exponentially on
  $\tfrac{1}{\gamma}$. We may use another argument to achieve a bound
  independent of $\gamma$ but exponentially
  dependent on $\tfrac{1}{\mu}$.  
\end{remark}

\begin{theorem}[Alternative reduced relative entropy bound]
  \label{the:alt-rre-bound}
  Let the conditions of Theorem \ref{the:rreb} hold, with the
  exception that  $\mu > 0$. We define the
  \emph{modified relative entropy} as
  \begin{equation}
    \eta_M(t) := \eta_R(t) + \frac{1}{2}\Norm{u - \hatu}_{\leb{2}(S^1)}^2 ,
  \end{equation}
  then 
  \begin{equation}
    \eta_M(t) 
    \leq 
    \qp{
      \eta_M(0) + \Norm{\mathfrak{R}}^2_{\leb{2}(S^1\times (0,t))}
    }\exp\qp{\int_0^t \widetilde{K}[\hatu](s) \d s},
  \end{equation}
  where
  \begin{equation}
    \widetilde{K}[\hatu](t)
    :=
    \max\qp{ 
      \frac{4}{3\mu} 
      \qp{\overline{\overline{W}}^2  + 1}
      ,
      2
    }
    \quad
    \AND \quad
    \overline{\overline{W}} = \Norm{W}_{\cont{2}(-\overline{M},\overline{M})},
  \end{equation}
  with $\overline{M}$ defined as in (\ref{eq:overlineM}).
\end{theorem}
\begin{proof}
  The equality of (\ref{eq:rre-1-gamma}) shows that
  \begin{equation}
    \label{eq:alt-rre-1}
    \begin{split}
      \dt \eta_R(t)
      &=
      \int_{S^1}
      \qp{v - \hatv} \qp{\pd x W'(\hatu) - \pd x W'(u) - \mathfrak{R}}
      -
      \frac{3}{4}\mu \qp{\pd {x} \hatv - \pd {x} v}^2
      \\
      &=
      \int_{S^1}
      \qp{\pd x v - \pd x \hatv} \qp{W'(u) - W'(\hatu)}
      +
      \qp{\hatv - v} \mathfrak{R}
      -
      \mu \qp{\pd {x} \hatv - \pd {x} v}^2
    \end{split}
  \end{equation}
  upon integrating by parts. We also see that
  \begin{equation}   \label{eq:alt-rre-2}
    \frac{1}{2} \dt \Norm{u - \hatu}_{\leb{2}(S^1)}^2
    =
    \int_{S^1} 
    \qp{u - \hatu} \qp{\pd t u - \pd t \hatu} 
    =
    \int_{S^1} 
    \qp{u - \hatu} \qp{\pd x v - \pd x \hatv},
  \end{equation}
  where we used (\ref{eq:perturbed-prob})$_1$.
  Taking the sum of (\ref{eq:alt-rre-1}) and (\ref{eq:alt-rre-2}) we
  have that
  \begin{equation}
    \begin{split}
      \dt
      \eta_M(t)
      &\leq
      \int_{S^1} 
      \qp{\pd x v - \pd x \hatv} \qp{W'(\hatu) - W'(u)}
      +
      \qp{\hatv - v} \mathfrak{R}
      \\
      & \qquad -
       \frac{3}{4}\mu \qp{\pd {x} \hatv - \pd {x} v}^2
      +
      \qp{u - \hatu} \qp{\pd x v - \pd x \hatv}.
    \end{split}
  \end{equation}
  Applying Young's inequality (with $\epsilon>0$) we see that
  \begin{equation}
    \begin{split}
      \dt
      \eta_M(t)
      &\leq
      \int_{S^1} 
      \qp{2\epsilon - \frac{3}{4} \mu}\qp{\pd x \hatv - \pd x v}^2
      +
      \frac{1}{4\epsilon}
      \qp{
        \qp{W'(\hatu) - W'(u)}^2
        + 
        \qp{u - \hatu}^2
      }
      \\
      &\qquad\qquad\qquad\qquad\qquad\qquad\qquad\qquad\qquad\qquad\qquad
      +
      \qp{v - \hatv}^2
      +
      \mathfrak R^2.
    \end{split}
  \end{equation}
  Choosing $\epsilon = \tfrac{3\mu}{8}$ we have
  \begin{equation}
    \begin{split}
      \dt \eta_M(t)
      &\leq
      \frac{2}{3\mu} 
      \qp{\overline{\overline{W}}^2+ 1}
      \Norm{u - \hatu}_{\leb{2}(S^1)}^2
      +
      \Norm{\mathfrak{R}}_{\leb{2}(S^1)}^2
      +
      \Norm{\hatv - v}_{\leb{2}(S^1)}^2
      \\
      &\leq 
      \max\qp{ 
        \frac{4}{3\mu} 
        \qp{\overline{\overline{W}}^2 + 1}
        ,
        2
      }
      \eta_M(t)
      +
      \Norm{\mathfrak{R}(\cdot,t)}_{\leb{2}(S^1)}^2.
    \end{split}
  \end{equation}
  Applying the Gronwall inequality from Lemma \ref{lem:gronwall}
  yields the desired result.
\end{proof}

\section{Discretisation and a posteriori setup}
\label{sec:discretisation}

In this section we describe the discretisation which we analyse for
the approximation of (\ref{eq:model-prob}). We show that the scheme
has a monotonically decreasing energy functional and that solutions to
the scheme exist and are bounded in terms of the initial datum. In
addition we introduce the necessary reconstruction operators on which
our a posteriori analysis relies.  
\begin{definition}[Finite element space]
  We discretise (\ref{eq:model-prob}) in space using a dG finite
  element method. To that end we let $S^1 :=[0,1]$ be the unit
  interval with matching endpoints and choose $0 = x_0 < x_1 < \dots <
  x_N = 1.$ We denote $K_i=[x_i,x_{i+1}]$ to be the $i$--th
  subinterval and let $h_i:= \norm{K_i}$ be its length with $\T{} = \{
  K_i \}_{i=0}^{N-1}$. We impose that the ratio $h_n/h_{n+1}$ is
  bounded from above and below for $n=0,\dots,N-1.$ We set $\E$ to be
  the set of common interfaces of $\T{}$. For $x_n \in \E$ we define
  $h_{\E}^-(x_n) := h_{n-1}$, $h_{\E}^+(x_n) := h_{n}$ and
  $h_{\E}(x_n) :=\frac{1}{2}(h_{n-1}+ h_{n})$ such that
  $h_{\E}^+,h_{\E}^-,h_{\E} \in \leb{\infty}(\E).$ Let $\poly p$ be
  the space of polynomials of degree less than or equal to $p$, then
  we introduce the \emph{finite element space}
  \begin{equation}
    \fes_p 
    :=
    \ensemble{\Phi : I \to \rR }
    { \Phi \vert _{K_i} \in \poly p{(K_i)}}.
  \end{equation}
\end{definition}

\begin{definition}[Broken Sobolev spaces]
  \label{defn:broken-sobolev-space}
  We introduce the broken Sobolev space
  \begin{equation}
    \sobh{k}(\T{})
    :=
    \ensemble{\phi}
             {\phi|_K\in\sobh{k}(K), \text{ for each } K \in \T{}}.
  \end{equation}
  We also make use of functions defined in these broken spaces
  restricted to the skeleton of the triangulation.
\end{definition}

\begin{definition}[Jumps and averages]
  \label{defn:averages-and-jumps}
  We define average, jump operators for arbitrary
  scalar functions $v\in\sobh{k}(\T{})$ as
  \begin{gather}
    \avg{v} :=  {\frac{1}{2}\qp{v^+ + v^-}}:= \frac{1}{2}\qp{ \lim_{s \searrow 0} v(\cdot + s) +\lim_{s \searrow 0} v(\cdot - s)}  
    ,
    \\
    \jump{v} := {\qp{v^- - v^+}}:=  \lim_{s \searrow 0} v(\cdot - s) -\lim_{s \searrow 0} v(\cdot + s).
  \end{gather}
  Note that $\jump{v},\avg{v} \in \leb{2}(\E).$
\end{definition}

We will often use the following Proposition which we state in full for
clarity but whose proof is merely using the identities in Definition
\ref{defn:averages-and-jumps}.
\begin{proposition}[Elementwise integration]
  \label{Pro:trace-jump-avg}
  For generic functions $\psi,\phi \in \sobh{1}(\T{})$ we have
  \begin{equation}
    \label{eq:jump-avg-eq1}
    \begin{split}
      \sum_{K\in\T{}}
      \int_K (\pd x \psi) \phi 
      =
      \sum_{K\in\T{}}
      \qp{
        -
        \int_K
        \psi \pd x \phi 
        +
        \int_{\partial K}
        \phi \psi n_K 
        },
    \end{split}
  \end{equation}
  where $n_K$ is the outward pointing unit normal to $\partial K.$
  Furthermore the following identity holds
  \begin{equation}
    \label{eq:jump-avg-eq2}
    \sum_{K\in\T{}}
    \int_{\partial K}
    \phi \psi
    n_K 
    =
    \int_\E 
    \jump{\psi} 
    \avg{\phi}
    +
    \int_{\E}
    {\jump{\phi}}
    \avg{\psi}
    =
    \int_{\E}
    \jump{\psi \phi}
.
  \end{equation}
\end{proposition}

\renewcommand{\enorm}[1]{\norm{#1}_{dG}}
\begin{definition}[Discrete norm]
  We introduce the broken $\sobh{1}(\T{})$ seminorm as
  \begin{equation}
    \enorm{u_h}^2 
    := 
    \sum_{K\in\T{}} \Norm{\pd x u_h}_{\leb{2}(K)}^2
    +   
    \Norm{ \sqrt{h_{\E}^{-1}} \jump{u_h}}_{\leb{2}(\E)}^2.
  \end{equation}
\end{definition}

\begin{definition}[Discrete gradient operators]
  \label{def:discrete-gradient}
  We define the discrete gradient operators $G^\pm:\sobh1\qp{\T{}} \to
  \fes_p$ such that
  \begin{equation}
    \int_{S^1} G^\pm[\psi] \Phi 
    =
    \sum_{K\in\T{}}\int_K \pd x \psi \Phi - \int_\E \jump{\psi} \Phi^\pm \Foreach \Phi\in\fes_p.
  \end{equation}
  Note that if $\psi\in\sobh1(S^1)$ then $G^\pm[\psi]$ is the $\leb{2}$
  projection of $\pd x \psi$ onto $\fes_p$.
\end{definition}

\begin{proposition}[Discrete integration by parts]
  \label{pro:discrete-int-by-parts}
  Given $G^\pm:\sobh1(\T{}) \to \fes_p$ we have that
  \begin{equation}
    \int_{S^1} G^\pm[\Psi] \Phi = -\int_{S^1} \Psi G^\mp[\Phi] \Foreach \Psi,\Phi\in\fes_p.
  \end{equation}
\end{proposition}
\begin{proof}
  The proof follows immediately from the definition of $G^\pm[\cdot]$ and
  the elementwise integration formulae in Proposition \ref{Pro:trace-jump-avg}. Indeed,
  \begin{equation}
    \begin{split}
      \int_{S^1} G^\pm[\Psi] \Phi 
      &=
      \sum_{K\in\T{}}\int_K \pd x \Psi \Phi - \int_\E \jump{\Psi} \Phi^\pm
      \\
      &=
      \sum_{K\in\T{}} - \int_K \Psi \pd x \Phi + \int_\E \jump{\Phi} \Psi^\mp
      \\
      &=
      -\int_{S^1} \Psi G^\mp[\Phi],
    \end{split}
  \end{equation}
  as required.
\end{proof}

\subsection{Discrete scheme}

We will examine the following class of semi-discrete numerical schemes
where we seek $\qp{u_h,v_h,\tau_h} \in \cont{1}(0,T;\rV_p)\times
\cont{1}(0,T;\rV_p)\times \cont{0}(0,T;\rV_p)$ determined such
that
\begin{equation}
  \label{eq:discrete-scheme}
  \begin{split}
    0 &= \int_{S^1} \pd t u_h \Phi - G^-[v_h] \Phi  \quad \Foreach \Phi \in \fes_p
    \\
    0 &= \int_{S^1} \pd t v_h \Psi - G^+[\tau_h] \Psi + \mu G^-[v_h] G^-[\Psi] \quad \Foreach \Psi \in \fes_p
    \\
    0 &= \gamma \bih{u_h}{\Zeta} + \int_{S^1} \tau_h \Zeta - W'(u_h) \Zeta \quad \Foreach \Zeta \in \fes_p,
  \end{split}
\end{equation}
where $\cA_h:\fes_p\times\fes_p \to\reals$ is a consistent, symmetric
bilinear form representing a discretisation of the Laplacian. We
impose that it is coercive with respect to the dG seminorm on
$\fes_p$. 

The initial data for the semi-discrete scheme are given as follows:
$u_h(\cdot,0)$ is the Ritz projection of $u_0$ with respect to
$\cA_h$, that is, $u_h$ satisfies
\begin{equation}
  \label{eq:uh-init}
  \bih{u_h(\cdot,0)}{\Phi} = \bih{u_0}{\Phi} \Foreach \Phi\in\fes_p,
\end{equation}
with the additional constraint of having the same mean value. We also
define $v_h(\cdot,0)$ to be the $\leb{2}$ projection of $v_0$, that is,
\begin{equation}
  \label{eq:vh-init}
  \int_{S^1} v_h(\cdot, 0) \Psi = \int_{S^1} v_0 \Psi \Foreach \Psi\in\fes_p.
\end{equation}


\begin{proposition}[Energy consistency]
  \label{pro:energy-consistency}
  The discretistaion (\ref{eq:discrete-scheme}) is energy consistent
  in the sense that there is a monotonically decreasing energy
  functional
  \cite{GiesselmannMakridakisPryer:2013,GiesselmannPryer:2013}. Let
  $(u_h,v_h,\tau_h)$ be a solution of \eqref{eq:discrete-scheme}, then
  the following equality holds for $0 < t <T$
  \begin{equation}
    \ddt \qp{\gamma \bih{u_h}{u_h} + \int_{S^1} W(u_h) + \frac{1}{2} \norm{v_h}^2}
    =
    - \mu \int_{S^1} \norm{G^-[v_h]}^2.
  \end{equation}
\end{proposition}
\begin{proof}
  Taking $\Phi = \tau_h$ in (\ref{eq:discrete-scheme})$_1$, $\Psi =
  v_h$ in (\ref{eq:discrete-scheme})$_2$ and taking the sum yields
  \begin{equation}
    \label{eq:energy-eq-1}
    0 
    =
    \int_{S^1} 
    \pd t u_h \tau_h
    -
    G^-[v_h] \tau_h 
    +
    \pd t v_h v_h
    -
    G^+[\tau_h] v_h 
    +
    \mu G^-[v_h] G^-[v_h].
  \end{equation}
  Taking $\Zeta = \pd t u_h$ in (\ref{eq:discrete-scheme})$_3$ and
  substituting into (\ref{eq:energy-eq-1}) we see
  \begin{equation}
    0 
    = 
    \gamma \bih{u_h}{\pd t u_h}
    +
    \int_{S^1} 
    W'(u_h) \pd t u_h
    +
    \pd t v_h v_h
    +
    \mu G^-[v_h] G^-[v_h],
  \end{equation}
  which gives
  \begin{equation}
    0 =
    \ddt \qp{
      \gamma \bih{u_h}{u_h}
      +
      \int_{S^1}
      W(u_h)
      +
      \frac{1}{2} v_h^2
    }
    +
    \mu \int_{S^1}
    |G^-[v_h]|^2,
  \end{equation}
  yielding the desired result.
\end{proof}

\begin{remark}[$\leb{\infty}$ bound of $u_h$]
  Proposition \ref{pro:energy-consistency} shows that the energy
  functional
  \begin{equation}
    \gamma
    \bih{u_h}{u_h} + \int_{S^1} W(u_h) + \frac{1}{2} \norm{v_h}^2 
  \end{equation}
  is non-increasing in time. Due to the coercivity of $\cA_h$
  this implies $\enorm{u_h(\cdot,t)}^2$ is uniformly bounded in
  time, in terms of the initial data $u_0,v_0$.  Thus we have that
  $\Norm{u_h}_{\leb{\infty}(S^1\times (0,T))}^2$ is bounded (in terms
  of the initial data) since the average value of $u_h$ is conserved in time.
\end{remark}

\begin{lemma}[Existence of solutions of \eqref{eq:discrete-scheme} {\cite[Lem 3.7]{GiesselmannPryer:2014}}]
  Solutions of \eqref{eq:discrete-scheme} with initial data defined in
  (\ref{eq:uh-init})--(\ref{eq:vh-init}) exist for arbitrarily long
  times $T$.
\end{lemma}

\begin{theorem}[A priori estimates for the scheme {\cite[Thm 6.1]{GiesselmannPryer:2014}}]
  \label{the:apriori}
  Let the exact solution $(u,v)$ of (\ref{eq:model-prob}) satisfy 
  \begin{equation}
    \begin{split}
      u &\in \cont{1}(0,T; \sobh{q+2}(S^1)) \cap \cont{0}(0,T;  \cont{q+3}(S^1)) \\
      v &\in \cont{1}(0,T;  \cont{q+2}(S^1)) \cap \cont{0}(0,T;  \cont{q+3}(S^1))
    \end{split}
  \end{equation}
  and let $W \in \cont{q+3}(\rR,[0,\infty)).$ In addition suppose $u_h, v_h, \tau_h$ are solutions of the semidiscrete scheme \eqref{eq:discrete-scheme}.
  Then there exists $C>0$ independent of $h,$ but depending on
  $q,T,\gamma, \frac{h}{\min_n h_n},u,v$ such that
  \begin{equation}
    \begin{split}
      \sup_{0 \leq t \leq T} \bigg(&\Norm{u_h(\cdot,t) - u(\cdot,t)}_{\operatorname{dG}} + \Norm{v_h(\cdot,t) - v(\cdot,t)}_{\leb{2}}\bigg) \\
      &\leq C h^q\bigg( \Norm{u}_{\leb{\infty}(0,T; \cont{q+3}(S^1))} + \Norm{v}_{\leb{\infty}(0,T; \cont{q+3}(S^1))}+\Norm{\pd t v}_{\leb{\infty}(0,T; \cont{q+2}(S^1))}\bigg) .
    \end{split}
  \end{equation}
\end{theorem}

\begin{remark}[Notation convention]
  To avoid the proliferation of constants, unless otherwise specified,
  we will henceforth use the convention that $a \lesssim b$ means $a
  \leq C b$ for a generic constant $C$ that may depend on the domain,
  triangulation or polynomial degree, but is \emph{independent} of the
  meshsize $h$, problem parameter $\mu$ and exact or
  discrete solutions $\qp{u,v}$, $\qp{u_h, v_h, \tau_h}$. We have also tried to clarify the dependency of the resultant estimator on $\gamma$,
  however, it is not feasible to make the constants completely independent of $\gamma$ in view of the $\gamma$
  dependency of $\overline M$ and $\overline W$ in Theorem \ref{the:rreb}. 
  Since the constants are not fully stated our final result will be an a posteriori \emph{indicator},
  however, an estimator can be achieved by explicitly tracking the
  constants, for clarity of exposition we will not do this here.
\end{remark}

To set up the a posteriori analysis we require access to two families of
reconstruction operators. 

\begin{definition}[Discrete reconstruction operators]
  \label{eq:discrete-reconstruction}
  We define $\D^\pm:\fes_p \to \fes_{p+1}$ to be the \emph{discrete
    reconstruction operator} satisfying for $\Psi\in\fes_p$
  \begin{equation}
    0 = \int_{S^1} \pd x \D^\pm[\Psi] \Phi - G^\pm[\Psi]\Phi \Foreach\Phi\in\fes_{p}
  \end{equation}
  and
  \[ \qp{ \D^\pm [\Psi]}^\pm = \Psi^\mp  \quad \text{ on } \E.\]
\end{definition}

\begin{remark}[Continuity of {$\D^\pm[\cdot]$}]
  \label{rem:continuity-and-appoximation-D}
  Note that $\D^\pm$ are constructed such that for any
  $\Psi\in\fes_p$ we have that $\D^\pm[\Psi]\in\fes_{p+1}\cap\cont{0}(S^1)$.
  In addition we have the following approximation properties, proofs
  of which can be found in \cite[c.f.]{MN06}
  \begin{gather}
    \Norm{\Psi - \D^\pm[\Psi]}_{\leb{2}(S^1)}^2 \lesssim  \Norm{\sqrt{h^\pm_\E}\jump{\Psi}}_{\leb{2}(\E)}^2
    \\
    \enorm{\Psi - \D^\pm[\Psi]}^2 \lesssim \Norm{\sqrt{h^{-1}_\E}\jump{\Psi}}_{\leb{2}(\E)}^2.
  \end{gather}
\end{remark}

\begin{remark}[Multidimensional extension]
  \label{rem:multid}
  The analysis which we present in this work is fully extendable to
  the multidimensional setting with the exception of the discrete
  reconstruction operators $\D^\pm$. The construction of appropriate
  generalisations of $\D^\pm$ is the topic of ongoing research, however, progress in this direction has been made in \cite{GeorgoulisHallMakridakis:2014}
  where the authors give an appropriate reconstruction for the case of two dimensional linear transport.
\end{remark}

\begin{remark}[Orthogonality]
  \label{rem:orthogonality-and-appoximation-D}
  Note that $\D^\pm$ are constructed such that for any
  $\Psi\in\fes_p$ and $\Phi \in \fes_{p-1}$ we have that 
\[ \int_{S^1} \qp{ D^\pm[\Psi] - \Psi} \Phi =0.\] 
  A  proof
can be found in \cite[c.f.]{GiesselmannMakridakisPryer:2014}
\end{remark}

\begin{definition}[Continuous projection operator]
  We define $P^C_{p} : \leb{2}(\T{}) \to \fes_{p}\cap \cont{0}(S^1)$ to
  be the $\leb{2}(S^1)$ orthogonal projection operator satisfying
  \begin{equation}
    \int_{S^1} P^C_p \qb{w} \Phi = \int_{S^1} w\Phi \Foreach \Phi\in\fes_{p}\cap \cont{0}(S^1).
  \end{equation}
  It is readily verifiable that $P^C_{p}$ is stable in $\leb{2}(S^1)$,
  that is, $\Norm{P^C_p \qb{w}}_{\leb{2}(S^1)}\leq \Norm{w}_{\leb{2}(S^1)}$
  and has optimal approximation properties
  \begin{equation}
    \Norm{P^C_p \qb{w} - w}_{\leb{2}(S^1)}
    \lesssim
    h^{p+1} \Norm{w}_{\sobh{q+1}(S^1)}.
  \end{equation}
\end{definition}

\begin{definition}[Continuous reconstruction operators]
  \label{def:ell-recon-operators}
  We define three continuous reconstruction operators,
  $\cR_1[u_h]\in\sobh{3}(S^1), \cR_2[u_h]\in\sobh{2}(S^1) \AND
  \cR[v_h]\in\sobh{2}(S^1)$ to be solutions of
  \begin{equation}
    \begin{split}
      0 &= \gamma \pd {xx} \cR_1[u_h] - P^C_{p+1} \qb{W'(u_h)} + \D^+[\tau_h]
      \\
      0 &= \gamma \pd {xx} \cR_2[u_h] - W'(u_h) + \tau_h
      \\
      0 &= \pd {xx} \cR[v_h] - \pd {xt} \cR_1[u_h],
    \end{split}
  \end{equation}
  respectively, such that each of the problems has matching mean value
  with the discrete solution, that is
  \begin{equation}
    0 = \int_{S^1} \cR_1[u_h] - u_h = \int_{S^1} \cR_2[u_h] - u_h = \int_{S^1} \cR[v_h] - v_h.
  \end{equation}
\end{definition}

\begin{lemma}[Reconstructed PDE system]
  \label{lem:reconstructed-pde-sys}
  The reconstructions given in Definition
  \ref{def:ell-recon-operators} satisfy the following perturbation of
  (\ref{eq:model-prob})
  \begin{equation}
    \begin{split}
      \pd t \cR_1[u_h] - \pd x \cR[v_h] &= 0
      \\
      \pd t \cR[v_h] - \pd x W'(\cR_1[u_h]) + \gamma \pd {xxx} \cR_1[u_h] - \mu \pd {xx} \cR[v_h]
      &= E,
    \end{split}
  \end{equation}
  where
  \begin{equation}
    E := \pd t {\qp{\cR[v_h] - v_h}} - \pd x {\qp{W'(\cR_1[u_h]) - P^C_{p+1} \qb{W'(u_h)}}}     
    -
    \mu \pd x {\qp{\pd t \cR_1[u_h] - \D^+[\pd t u_h]}}.
  \end{equation}
\end{lemma}
\begin{proof}
  Using the smoothness of the reconstruction $\cR_1[u_h]$ we see that
  \begin{equation}
    \label{eq:reconstructed-pde-1}
    0 = \gamma \pd {xxx} \cR_1[u_h] - \pd x P^C_{p+1} \qb{W'(u_h)} + \pd x \D^+[\tau_h].
  \end{equation}
  Using the semi-discrete scheme (\ref{eq:discrete-scheme})$_1$ we have
  that $G^-[v_h] = \pd t u_h$, substituting this into
  (\ref{eq:discrete-scheme})$_2$ we see
  \begin{equation}
    \begin{split}
      0 
      &=
      \int_{S^1} \pd t v_h \Psi - G^+[\tau_h] \Psi + \mu G^-[v_h] G^-[\Psi]
      \\
      &=
      \int_{S^1} \pd t v_h \Psi - G^+[\tau_h] \Psi + \mu \pd t u_h G^-[\Psi].
    \end{split}
  \end{equation}
  Making use of the discrete integration by parts in Proposition
  \ref{pro:discrete-int-by-parts} we have that
  \begin{equation}
    0 = 
    \int_{S^1} \pd t v_h \Psi - G^+[\tau_h] \Psi - \mu G^+[\pd t u_h] \Psi.
  \end{equation}
  Now in view of the discrete reconstruction operator given in
  Defintion \ref{eq:discrete-reconstruction} we see
  \begin{equation}
    \label{eq:reconstructed-pde-2}
    0 = 
    \int_{S^1} \pd t v_h \Psi - G^+[\tau_h] \Psi - \mu \pd x \D^+[\pd t u_h] \Psi.
  \end{equation}
  As $\pd t v_h(\cdot,t) , G^+[\tau_h](\cdot,t), \pd x D^+[\pd t
  u_h](\cdot,t) \in \rV_p$ we may write (\ref{eq:reconstructed-pde-2})
  pointwise as
  \begin{equation}
    \label{eq:reconstructed-pde-3}
    0 = 
    \pd t v_h - G^+[\tau_h] - \mu \pd x \D^+[\pd t u_h].
  \end{equation}
  Using (\ref{eq:reconstructed-pde-1}) and
  (\ref{eq:reconstructed-pde-3}) together with the definition of
  $\cR[v_h]$ we see that 
  \begin{equation}
    \begin{split}
      0
      &=
      \pd t \cR[v_h] 
      -
      \pd x W'(\cR_1[u_h])
      +
      \gamma \pd {xxx} \cR_1[u_h] 
      -
      \mu \pd {xx} \cR[v_h] 
      \\
      &
      \qquad 
      -
      \pd t {\qp{\cR[v_h] - v_h}} 
      +
      \pd x {\qp{W'(\cR_1[u_h]) - P^C_{p+1} \qb{W'(u_h)}}}
      +
      \mu \pd x {\qp{\pd t \cR_1[u_h] - \D^+[\pd t u_h]}},
    \end{split}
  \end{equation}
  showing the second equation of the Lemma, the first is obtained
  using the definition of $\cR[v_h]$, concluding the proof.
\end{proof}

\begin{remark}[Regularity bounds for the reconstructions]
  \label{rem:properties-of-ell-recon}
  Note that the problems which define the reconstruction operators in
  Defintion \ref{def:ell-recon-operators} are well posed in view of
  the elliptic problem's unique solvability, moreover, thanks to
  elliptic regularity, we have
  \begin{equation}
    \begin{split}
      \Norm{\cR_1[u_h]}_{\sobh{k+1}(S^1)} 
      &\lesssim 
      \frac{1}{\gamma} \Norm{P^C_{p+1} \qb{W'(u_h)} - \D^+[\tau_h]}_{\sobh{k-1}(S^1)} 
      \Foreach k \in \{0,1,2\}
      \\
      \Norm{\cR_2[u_h]}_{\sobh{k+1}(S^1)} 
      &\lesssim 
      \frac{1}{\gamma} \Norm{ W'(u_h) - \tau_h}_{\sobh{k-1}(S^1)} \Foreach k \in \{0,1\} \AND
      \\
      \Norm{\cR[v_h]}_{\sobh{k+1}(S^1)} 
      &\lesssim 
      \Norm{\pd {xt} \cR_1[u_h]}_{\sobh{k-1}(S^1)}  \Foreach k \in \{0,1\}.
    \end{split}
  \end{equation}
\end{remark}

\newtheorem{Hyp}[theorem]{Assumption}

\begin{Hyp}[A posteriori control on {$\bih{\cdot}{\cdot}$}]
  \label{ass:a posteriori-control}
  The reconstruction $\cR_2[u_h]$ is the \emph{elliptic reconstruction}
  of $u_h$ \cite[c.f.]{Makridakis:2003}. We will make the assumption that there exists an optimal
  order elliptic a posteriori estimate controlling $\enorm{u_h -
    \cR_2[u_h]}$, that is, there exists a functional $\Eta_1$ depending
  only upon $u_h$ and the problem data such that
  \begin{equation}
    \begin{split}
      \enorm{u_h - \cR_2[u_h]}
      &\lesssim 
      \Eta_1[u_h, \frac{1}{\gamma}\qp{\tau_h - W'(u_h)}] \sim \Oh(h^{p}).
    \end{split}
  \end{equation}
\end{Hyp}

\begin{example}[A posteriori control for the interior penalty discretisation]
  \label{ex:ip-apost}
  Taking $f := \tau_h - W'(u_h)$, if $\bih{\cdot}{\cdot}$ takes the
  form of an interior penalty discretisation
  \begin{equation}
    \bih{u_h}{\Zeta} 
    = 
    \int_{S^1} \pd x u_h \pd x \Zeta
    -
    \int_\E \jump{u_h}\avg{\pd x \Zeta} + \jump{\Zeta}\avg{\pd x u_h}
    -
    \frac{\sigma}{h_\E} \jump{u_h}\jump{\Zeta},
  \end{equation}
  where $\sigma$ is the penalty parameter and is chosen large enough to
  guarantee coercivity, we may use estimates of the form
  \begin{equation}
    \Eta_1[u_h, f]^2
    = 
    \sum_{K\in\T{}}
    h_K^2 \Norm{f - \pd {xx} u_h}_{\leb{2}(K)}^2
    +
    \sum_{e\in\E}\qb{h_e \Norm{\jump{\pd x u_h}}_{\leb{2}(e)}^2
    +
    \sigma^2 h_e^{-1} \Norm{\jump{u_h}}_{\leb{2}(e)}^2}.
  \end{equation}
  See for example \cite[Thm 3.1]{KarakashianPascal:2003}.
\end{example}

\section{A posteriori analysis}
\label{sec:a posteriori}

We begin this section by stating some technical Lemmata required for
the main result.

\begin{lemma}[Reduced relative entropy bound]
  \label{lem:rre-bound-discrete}
  Let $$\qp{u,v}\in\cont{1}(0,T;\sobh1(S^1))\cap\cont{0}(0,T;\sobh3(S^1))\times
  \cont{1}(0,T;\leb{2}(S^1))\cap\cont{0}(0,T;\sobh2(S^1))$$ solve the model problem
  (\ref{eq:model-prob}) and $\qp{u_h,v_h,\tau_h}\in
  \cont{1}([0,T),\rV_p)\times \cont{1}([0,T),\rV_p)\times
  \cont{0}([0,T),\rV_p)$ be the semidiscrete approximations generated
  by the scheme (\ref{eq:discrete-scheme}) then given the reduced
  relative entropy 
  \begin{equation}
    \eta_R(t) 
    :=
    \int_{S^1} 
    \frac{\gamma}{2} \qp{\pd x u - \pd x \cR_1[u_h]}^2
    +
    \frac{1}{2} \qp{v - \cR[v_h]}^2
    +
    \frac{\mu}{4} \int_0^t \norm{v - \cR[v_h]}_{\sobh{1}(S^1)},
  \end{equation}
  we have that
  \begin{equation}
    \ddt \eta_R(t) \leq K\big[\cR_1[u_h(\cdot,t)]\big] \eta_R(t) + \Norm{E(\cdot,t)}^2_{\leb{2}(S^1)},
  \end{equation}
  with $E = E_1 - E_2 - E_3$ and
  \begin{gather}
    E_1 := \pd t {\qp{\cR[v_h] - v_h}}
    \\
    E_2 := \pd x {\qp{W'(\cR_1[u_h]) - P^C_{p+1} \qb{W'(u_h)}}}
    \\
    E_3 := \mu \pd x {\qp{\pd t \cR_1[u_h] - \D^+[\pd t u_h]}}.
  \end{gather}
\end{lemma}
\begin{proof}
  The proof consists of taking $\hatu = \cR_1[u_h]$ and $\hatv =
  \cR[v_h]$ in Theorem \ref{the:rreb}. Noting that from Lemma
  \ref{lem:reconstructed-pde-sys} the reconstructions satisfy the
  correct PDE with residual $\mathfrak{R} = E$.
\end{proof}

\begin{lemma}[Modified relative entropy bound]
  Let the conditions of Lemma \ref{lem:rre-bound-discrete} hold. Given
  the modified relative entropy
  \begin{equation}
    \eta_M(t) 
    :=
    \int_{S^1} 
    \frac{\gamma}{2} \qp{\pd x u - \pd x \cR_1[u_h]}^2
    +
    \frac{1}{2} \qp{u - \cR_1[u_h]}^2
    +
    \frac{1}{2} \qp{v - \cR[v_h]}^2
    +
    \frac{\mu}{4} \int_0^t \norm{v - \cR[v_h]}_{\sobh{1}(S^1)},
  \end{equation}
  we have that
  \begin{equation}
    \ddt \eta_M(t) \leq \widetilde K\big[\cR_1[u_h(\cdot,t)]\big] \eta_M(t) + \Norm{E(\cdot,t)}^2_{\leb{2}(S^1)},
  \end{equation}
  with $E$ given in Lemma \ref{lem:rre-bound-discrete}.
\end{lemma}
\begin{proof}
The proof is analogous to that of Lemma  \ref{lem:rre-bound-discrete} using Theorem 
\ref{the:alt-rre-bound} instead of Theorem \ref{the:rreb}.
\end{proof}

\begin{lemma}[Bound on the reconstruction of $u_h$]
  \label{lem-bd-on-Ruh}
  Let $u_h,v_h,\tau_h$ be given by \eqref{eq:discrete-scheme} and $\cR_1[u_h]$ be the reconstruction given in
  Definition \ref{def:ell-recon-operators}, then 
  \begin{multline}
      \Norm{\cR_1[u_h] - u_h}_{\leb{2}(S^1)} 
      \lesssim
      \enorm{\cR_1[u_h] - u_h}
      \lesssim
      \Eta_1[u_h, \frac{1}{\gamma}\qp{\tau_h - W'(u_h)}]
      \\
     + \frac{1}{\gamma}\Norm{P^C_{p+1} \qb{W'(u_h)} - W'(u_h)}_{\sobh{-1}(S^1)}
     +
      \frac{1}{\gamma} \Norm{\D^+[\tau_h] - \tau_h}_{\sobh{-1}(S^1)}.
    \end{multline}
\end{lemma}
\begin{proof}
  Using the triangle inequality we have that
  \begin{equation}
    \label{eq:bd-on-Ruh-1}
    \begin{split}
      \enorm{\cR_1[u_h] - u_h}
      &\lesssim
      \norm{{\cR_1[u_h] - \cR_2[u_h]}}_1
      +
      \enorm{{\cR_2[u_h] - u_h}}.
    \end{split}
  \end{equation}
  Using  the elliptic regularity of the problem
  we have that
  \begin{equation}
    \label{eq:bd-on-Ruh-2}
    \begin{split}
      \norm{{\cR_1[u_h] - \cR_2[u_h]}}_1
      &\lesssim
      \frac{1}{\gamma}\Norm{{{W'(u_h)-P^C_{p+1} \qb{W'(u_h)} + \tau_h - \D^+[\tau_h]}}}_{\sobh{-1}(S^1)}
      \\
      &\lesssim
      \frac{1}{\gamma}\Norm{{{W'(u_h)-P^C_{p+1} \qb{W'(u_h)}}}}_{\sobh{-1}(S^1)} 
      + \frac{1}{\gamma}\Norm{{{\tau_h - \D^+[\tau_h]}}}_{\sobh{-1}(S^1)}.
    \end{split}
  \end{equation}
  The result then follows  from Assumption
  \ref{ass:a posteriori-control} since
  \begin{equation}
    \label{eq:bd-on-Ruh-3}
      \enorm{\cR_2[u_h] - u_h} 
      \lesssim
      \Eta_1[u_h, \frac{1}{\gamma}\qp{\tau_h - W'(u_h)}].
  \end{equation}
  Substituting (\ref{eq:bd-on-Ruh-2}) and (\ref{eq:bd-on-Ruh-3}) into
  (\ref{eq:bd-on-Ruh-1}) concludes the proof.
\end{proof}

\begin{lemma}[Bounds on the reconstruction of $v_h$]
  \label{lem-bd-on-Rvh}
  Let $u_h,v_h,\tau_h$ be given by \eqref{eq:discrete-scheme} and $\cR[v_h]$ be the reconstruction given in
  Definition \ref{def:ell-recon-operators}, then 
  \begin{equation}\label{e1-bd-on-Rvh}
    \begin{split}
      \Norm{\cR[v_h] - v_h}_{\leb{2}(S^1)} 
      &\lesssim
      \Eta_1[\pd t u_h, \frac{1}{\gamma}\qp{\pd t \tau_h - \pd t W'(u_h)}]
      \\
      &\qquad +
       \frac{1}{\gamma}\Norm{P^C_{p+1} \qb{\pd t W'(u_h)} - \pd t W'(u_h)}_{\sobh{-1}(S^1)}
      \\
      &
      \qquad +
      \frac{1}{\gamma}\Norm{\D^+[\pd t \tau_h] - \pd t \tau_h}_{\sobh{-1}(S^1)}
      +
      \Norm{\sqrt{h_\E^-}\jump{v_h}}_{\leb{2}(\E)}
    \end{split}
  \end{equation}
  and
    \begin{equation}\label{e2-bd-on-Rvh}
    \begin{split}
  \norm{\cR[v_h] - v_h}_{\operatorname{dG}} 
      &\lesssim
      \Eta_1[\pd t u_h, \frac{1}{\gamma}\qp{\pd t \tau_h - \pd t W'(u_h)}]
      \\
      &\qquad +
       \frac{1}{\gamma}\Norm{P^C_{p+1} \qb{\pd t W'(u_h)} - \pd t W'(u_h)}_{\sobh{-1}(S^1)}
      \\
      &
      \qquad +
      \frac{1}{\gamma}\Norm{\D^+[\pd t \tau_h] - \pd t \tau_h}_{\sobh{-1}(S^1)}
      +
      \Norm{\sqrt{h_\E^{-1}}\jump{v_h}}_{\leb{2}(\E)}
    \end{split}
  \end{equation}
\end{lemma}
\begin{proof}
We firstly prove \eqref{e1-bd-on-Rvh}.
  Using the triangle inequality we have that
  \begin{equation}
    \label{eq:bd-on-Rvh-1}
    \begin{split}
      \Norm{{{\cR[v_h] - v_h}}}_{\leb{2}(S^1)}
      &\lesssim
      \Norm{{{\cR[v_h] - \D^-[v_h]}}}_{\leb{2}(S^1)}
      +
      \Norm{{{\D^-[v_h] - v_h}}}_{\leb{2}(S^1)}.
    \end{split}
  \end{equation}
  Using a Poincar\'e inequality and the approximation properties of
  $\D^-$ given in Remark \ref{rem:continuity-and-appoximation-D}, we see
  \begin{equation}
    \label{eq:bd-on-Rvh-2}
    \begin{split}
      \Norm{{{\cR[v_h] - v_h}}}_{\leb{2}(S^1)}
      &\lesssim
      \Norm{\pd {x} {\qp{\cR[v_h] - \D^-[v_h]}} }_{\leb{2}(S^1)}
      +
    \Norm{\sqrt{h_\E^-}\jump{v_h}}_{\leb{2}(\E)}
      \\
      &\lesssim
      \Norm{\pd t {\qp{\cR_1[u_h] - u_h}}}_{\leb{2}(S^1)}
      +
    \Norm{\sqrt{h_\E^-}\jump{v_h}}_{\leb{2}(\E)}.
    \end{split}
  \end{equation}
  Invoking the result of Lemma \ref{lem-bd-on-Ruh} gives
  \begin{equation}
    \label{eq:bd-on-Rvh-3}
    \begin{split}
      \Norm{\pd t {\qp{\cR_1[u_h] - u_h}}}_{\leb{2}(S^1)} 
      &\lesssim
      \Eta_1[\pd t u_h, \frac{1}{\gamma}\qp{\pd t {\qp{\tau_h - W'(u_h)}}}]
      \\
      &
      \qquad +
      \frac{1}{\gamma}\Norm{\pd t {\qp{P^C_{p+1} \qb{W'(u_h)} - W'(u_h)}}}_{\sobh{-1}(S^1)}
      \\
      &\qquad +
      \frac{1}{\gamma}\Norm{\pd t {\qp{\D^+[\tau_h] - \tau_h}}}_{\sobh{-1}(S^1)}.
    \end{split}
  \end{equation}
  Substituting (\ref{eq:bd-on-Rvh-3}) into (\ref{eq:bd-on-Rvh-2})
  gives \eqref{e1-bd-on-Rvh}.
  Equation \eqref{e2-bd-on-Rvh} follows from
  \begin{multline}
    \label{eq:bd-on-Rvh-4}
      \norm{{{\cR[v_h] - v_h}}}_{\operatorname{dG}}
      \leq
      \norm{{{\cR[v_h] - \D^-[v_h]}}}_{\sobh{1}(S^1)}
      +
      \norm{{{\D^-[v_h] - v_h}}}_{\operatorname{dG}}\\
      \lesssim
      \Norm{\pd t {\qp{\cR_1[u_h] - u_h}}}_{\leb{2}(S^1)}
      + 
      \Norm{\sqrt{h_\E^{-1}}\jump{v_h}}_{\leb{2}(\E)}
    \end{multline}
  and \eqref{eq:bd-on-Rvh-3}
\end{proof}

\begin{lemma}[Upper bound on {$E_1$}]
  \label{lem:control-of-E1}
  Let the conditions of Lemma \ref{lem:rre-bound-discrete} hold, then
  \begin{equation}
    \begin{split}
      \Norm{E_1}^2_{\leb{2}(S^1)}
      &\lesssim
      \Eta_1[\pd {tt} u_h, \frac{1}{\gamma} \qp{\pd {tt} \tau_h - \pd {tt} W'(u_h)}]^2
      \\
      &\qquad +
      h
      \qp{
        \frac{1}{\gamma^2}\Norm{\jump{\pd {tt} \tau_h}}^2_{\leb{2}(\E)}
        +
        \Norm{\jump{\pd {t} v_h}}^2_{\leb{2}(\E)}
        +
        \frac{1}{\gamma^2}
        \Norm{\jump{\pd {tt} W'(u_h)}}^2_{\leb{2}(\E)}
      }
      \\
      &\qquad\qquad +
      \frac{1}{\gamma^2}\sum_{K\in\T{}} h^{2p+2}_K \Norm{\pd {tt} W'(u_h)}^2_{\sobh{p+1}(K)}.
    \end{split}
  \end{equation}
\end{lemma}
\begin{proof}
  Using the triangle inequality we have that
  \begin{equation}
    \begin{split}
      \Norm{E_1}^2_{\leb{2}(S^1)}
      &=
      \Norm{\pd t {\qp{\cR[v_h] - v_h}}}^2_{\leb{2}(S^1)}
      \\
      &\lesssim
      \Norm{\pd t {\qp{\cR[v_h] - \D^-[v_h]}}}^2_{\leb{2}(S^1)}
      +
      \Norm{\pd t {\qp{\D^-[v_h] - v_h}}}^2_{\leb{2}(S^1)}.
    \end{split}
  \end{equation}
  Using a Poincar\'e inequality and the approximation properties of
  $\D^-$ given in Remark \ref{rem:continuity-and-appoximation-D}, we see
  \begin{equation}
    \begin{split}
      \Norm{E_1}^2_{\leb{2}(S^1)}
      &\lesssim
      \Norm{\pd {tx} {\qp{\cR[v_h] - \D^-[v_h]}}}^2_{\leb{2}(S^1)}
      +
       h \Norm{\jump{\pd t v_h}}^2_{\leb{2}(\E)}
      \\
      &\lesssim
      \Norm{\pd {tt} {\qp{\cR_1[u_h] - u_h}}}^2_{\leb{2}(S^1)}
      +
      h\Norm{\jump{\pd t v_h}}^2_{\leb{2}(\E)}.
    \end{split}
  \end{equation}
  In view of Lemma \ref{lem-bd-on-Ruh} we have that
  \begin{equation}
    \label{eq:bound-for-E1}
    \begin{split}
      \Norm{E_1}^2_{\leb{2}(S^1)}
      &\lesssim
      \frac{1}{\gamma^2}
      \Norm{\pd {tt} {\qp{W'(u_h)-P^C_{p+1} \qb{W'(u_h)}}}}^2_{\sobh{-1}(S^1)}
      +
      h\Norm{\jump{\pd t v_h}}^2_{\leb{2}(\E)}
      \\
      &\qquad +
      \Eta_1[\pd {tt} u_h, \frac{1}{\gamma}\qp{\pd {tt} \tau_h - \pd {tt} W'(u_h)}]^2
      +
      \frac{h}{\gamma^2}
      \Norm{\jump{\pd {tt} \tau_h}}^2_{\leb{2}(\E)}.
    \end{split}
  \end{equation}
  Notice the bound is already a posteriori computable, however to avoid
  the computation of \\$P^C_{p+1} \qb{W'(u_h)}$ we give a bound for this
  term. To that end let $S_p : {\sobh{1}(\T{})}\to\fes_p$ be a
  projection operator defined such that
  \begin{equation}
    \begin{split}
      \int_{S^1} S_p \qb{w} \Phi &= \int_{S^1} w \Phi \Foreach \Phi\in\fes_{p-2}
      \\
      S_p \qb{w}(x_n^\pm) &= w(x_n^\pm).
    \end{split}
  \end{equation}
  Note that $S_p$ exactly projects piecewise polynomials of degree $p$, hence we
  have the approximation result
  \begin{equation}
    \label{eq:approx-of-S}
    \Norm{w - S_p\qb{w}}_{\leb{2}(S^1)} ^2
    \lesssim
    \sum_{K\in\T{}} h_K^{2p+2} \norm{w}_{\sobh{p+1}(K)}^2
  \end{equation}
  and
  \begin{equation}
    \enorm{w - S_p \qb{w}}^2
    \lesssim
    \sum_{K\in\T{}} h_K^{2p} \norm{w}_{\sobh{p+1}(K)}^2.
  \end{equation}
  Now
  \begin{equation}
    \begin{split}
      \Norm{\pd {tt} {\qp{W'(u_h)-P^C_{p+1} \qb{W'(u_h)}}}}^2_{\sobh{-1}(S^1)}
      &\leq
      \Norm{\pd {tt} {\qp{W'(u_h)-S_p \qb{W'(u_h)}}}}^2_{\leb{2}(S^1)}
      \\
      &\quad +
      \Norm{\pd {tt} {\qp{S_p \qb{W'(u_h)} - P^C_{p+1} \big[S_p \qb{W'(u_h)}\big]}}}^2_{\leb{2}(S^1)}
      \\
      &\quad +
      \Norm{\pd {tt} {\qp{ P^C_{p+1} \qb{S_p \qb{W'(u_h)}} - P^C_{p+1} \qb{W'(u_h)}}}}^2_{\leb{2}(S^1)}
      \\
      &\quad 
      =: 
      E^{1,1} + E^{1,2} + E^{1,3}.
    \end{split}
  \end{equation}
  In view of the stability of the $\leb{2}$ projection we have that 
  \begin{equation}
    E^{1,3} \leq E^{1,1}.
  \end{equation}
  From the approximation properties of $S_p$ given in
  (\ref{eq:approx-of-S}) we have 
  \begin{equation}
    E^{1,1} 
    \lesssim  
    \sum_K h_K^{2p+2} \Norm{\pd {tt} W'(u_h)}^2_{\sobh{p+1}(K)}.
  \end{equation}
  Moreover,
  \begin{equation}
    \begin{split}
      E^{1,2}
      &\lesssim
      h\Norm{\jump{\pd {tt} W'(u_h)}}^2_{\leb{2}(\E)}
    \end{split}
  \end{equation}
  hence
  \begin{equation}
    \label{eq:final-bound-for-Eij}
    \begin{split}
      \Norm{\pd {tt} {\qp{W'(u_h)-P^C_{p+1} \qb{W'(u_h)}}}}^2_{\sobh{-1}(S^1)}
      &\lesssim
      h\Norm{\jump{\pd {tt} W'(u_h)}}^2_{\leb{2}(\E)}
      +
      \sum_K h_K^{2p+2} \Norm{\pd {tt} W'(u_h)}^2_{\sobh{p+1}(K)}.
    \end{split}
  \end{equation}
  Combining (\ref{eq:bound-for-E1}) with
  (\ref{eq:final-bound-for-Eij}) yields the desired result.
\end{proof}

\begin{lemma}[Upper bound on {$E_2$}]
  \label{lem:control-of-E2}
  Let the conditions of Lemma \ref{lem:rre-bound-discrete} hold, then
  \begin{equation}
    \begin{split}
      \Norm{E_2}^2_{\leb{2}(S^1)}
      &\lesssim
      \Eta_1[u_h, \frac{1}{\gamma} \qp{\tau_h - W'(u_h)}]^2
      +
      \frac{h}{\gamma^2}
      \qp{\Norm{\jump{\tau_h}}^2_{\leb{2}(\E)}
        +
        \Norm{\jump{u_h}}^2_{\leb{2}(\E)}
      }
      \\
      &\qquad +
      \Norm{\sqrt{h_\E^{-1}}\jump{u_h}}^2_{\leb{2}(\E)}
      \\
      &\qquad 
      +
      \sum_{K\in\T{}} h_K^{2p} \Norm{W'(u_h)}^2_{\sobh{p+1}(K)} 
      +
      \frac{1}{\gamma^2} \sum_{K\in\T{}} h_K^{2p+2} \Norm{W'(u_h)}^2_{\sobh{p+1}(K)}.
    \end{split}
  \end{equation}      
\end{lemma}

\begin{proof}
  To begin we note that in view of the triangle inequality
  \begin{equation}
    \label{eq:bound-on-E2}
    \begin{split}
      \Norm{E_2}^2_{\leb{2}(S^1)}
      &= 
      \Norm{\pd x {\qp{W'(\cR_1[u_h]) - P^C_{p+1} \qb{W'(u_h)}}}}^2_{\leb{2}(S^1)}
      \\
      &\lesssim
      \enorm{ W'(\cR_1[u_h]) -  W'(u_h)}^2
      +
      \enorm{ W'(u_h) - P^C_{p+1} \qb{W'(u_h)}}^2
      \\
      &=:
      E^{2,1} + E^{2,2}.
    \end{split}
  \end{equation}
  To control the term $E^{2,1}$ we note
  \begin{equation}
      E^{2,1} 
      =
      \enorm{ W'(\cR_1[u_h]) -  W'(u_h)}^2
      \lesssim
      \enorm{ \cR_1[u_h] - u_h}^2.
  \end{equation}
  Applying Lemma \ref{lem-bd-on-Ruh} and the same principles as in the
  proof of Lemma \ref{lem:control-of-E1} we arrive at
  \begin{equation}
    \label{eq:bound-on-E21}
    \begin{split}
      E^{2,1}
      &\lesssim
      \Eta_1[u_h, \frac{1}{\gamma}\qp{\tau_h - W'(u_h)}]^2
      +
      \frac{h}{\gamma^2}\qp{\Norm{\jump{\tau_h}}^2_{\leb{2}(\E)}
        +
        \Norm{\jump{u_h}}^2_{\leb{2}(\E)}}
      \\
      &\qquad +
      \frac{1}{\gamma^2} \sum_K h_K^{2p+2} \Norm{W'(u_h)}^2_{\sobh{p+1}(K)}.
    \end{split}
  \end{equation}
  To bound the term $E^{2,2}$ we reuse the methodology used to control
  $E^{1,1}$ given in the proof of Lemma \ref{lem:control-of-E1}. We
  have
  \begin{equation}
    \label{eq:bound-on-E22}
    \begin{split}
      E^{2,2}
      &=
      \enorm{ W'(u_h) - P^C_{p+1} \qb{W'(u_h)}}^2
      \\
      &\lesssim
      \enorm{{{W'(u_h)-S_p \qb{W'(u_h)}}}}^2
      +
      \enorm{{{S_p \qb{W'(u_h)} - P^C_{p+1}\qb{ S_p \qb{W'(u_h)}}}}}^2
      \\
      &\qquad +
      \enorm{{{ P^C_{p+1} \qb{S_p \qb{W'(u_h)}} - P^C_{p+1} \qb{W'(u_h)}}}}^2
      \\
      &\lesssim
     \Norm{\sqrt{h_\E^{-1}}\jump{u_h}}_{\leb{2}(\E)}^2 
      +
      \sum_{K\in\T{}} h_K^{2p} \Norm{W'(u_h)}^2_{\sobh{p+1}(K)},
    \end{split}
  \end{equation}
  in view of the stability of the $\leb{2}$ projection in
  $\sobh{1}(\T{})$ and inverse inequalities. Inserting
  (\ref{eq:bound-on-E21}) and (\ref{eq:bound-on-E22}) into
  (\ref{eq:bound-on-E2}) concludes the proof.
\end{proof}

\begin{lemma}[Upper bound on {$E_3$}]
  \label{lem:control-of-E3}
  Let the conditions of Lemma \ref{lem:rre-bound-discrete} hold, then
  \begin{equation}
    \begin{split}
      \Norm{E_3}^2_{\leb{2}(S^1)}
      &\lesssim
      \mu \bigg[
      \Eta_1[\pd t u_h, \frac{1}{\gamma}\qp{\pd t \tau_h - \pd t W'(u_h)}]^2
      +
     \frac{ h}{\gamma^2} 
      \qp{
        \Norm{\jump{\pd t \tau_h}}^2_{\leb{2}(\E)}
        +
        \Norm{\jump{\pd t W'(u_h)}}^2_{\leb{2}(\E)}
      }
      \\
      &\qquad \qquad \qquad \qquad 
      +
      \Norm{\sqrt{h_\E^{-1}}\jump{\pd t u_h}}^2_{\leb{2}(\E)}
      +
      \frac{1}{\gamma^2}\sum_{K\in\T{}} h^{2p+2}_K \Norm{\pd t W'(u_h)}^2_{\sobh{p+1}(K)}
      \bigg ].
    \end{split}
  \end{equation}      
\end{lemma}
\begin{proof}
  In view of the triangle inequality we have 
  \begin{equation}
    \begin{split}
      \Norm{E_3}_{\leb{2}(S^1)}^2
      &=
      \mu \Norm{\pd x {\qp{\pd t \cR_1[u_h] - \D^+[\pd t u_h]}}}_{\leb{2}(S^1)}^2
      \\
      &\lesssim
      \mu
      \qp{\enorm{\pd t \cR_1[u_h] - \pd t u_h}^2
        +
        \enorm{\pd t u_h - \D^+[\pd t u_h]}^2
      }
      \\
      &=:
      E^{3,1} + E^{3,2}.
    \end{split}
  \end{equation}
  Applying Lemma \ref{lem-bd-on-Ruh} to $E^{3,1}$ and the
  approximation properties of $\D^+$ concludes the proof.
\end{proof}

\begin{theorem}[A posteriori control of the reduced relative entropy]
  \label{the:apost-rre}
    Let the conditions of Lemma \ref{lem:rre-bound-discrete} hold, then
    \begin{equation}
      \eta_R(t) \lesssim \qp{ \eta_R(0) +
      \int_0^t \mathfrak{E}_s[u_h(s),v_h(s),\tau_h(s)]^2\d s}  \exp\qp{\int_0^t K[\cR_1[u_h]](s) \d s} 
    \end{equation}
    with 
    \begin{equation}
      \label{eq:estimator-et}
      \begin{split}
        \mathfrak{E}_t [u_h,v_h,\tau_h]^2
        &:=
        \Eta_1[u_h, \frac{1}{\gamma}\qp{\tau_h - W'(u_h)}]^2
        +
        \mu\Eta_1[\pd t u_h, \frac{1}{\gamma}\qp{\pd t \tau_h - \pd t W'(u_h)}]^2
        \\
        &\qquad +
        \Eta_1[\pd {tt} u_h, \frac{1}{\gamma}\qp{\pd {tt} \tau_h - \pd {tt} W'(u_h)}]^2
        +
        \Norm{\sqrt{h_\E^{-1}}\jump{u_h}}_{\leb{2}(\E)}^2
        \\
        &\qquad
        +
        \mu\Norm{\sqrt{h_\E^{-1}}\jump{\pd t u_h}}_{\leb{2}(\E)}^2
        +
        \sum_{K\in\T{}} h_K^{2p} \Norm{W'(u_h)}_{\sobh{p+1}(K)}^2
        \\
        &\qquad
        +
        \frac{h}{\gamma^2}
        \qp{
          \Norm{\jump{\tau_h}}_{\leb{2}(\E)}^2
          +
          \Norm{\jump{\pd {tt} \tau_h}}_{\leb{2}(\E)}^2
          +
          \Norm{\jump{u_h}}_{\leb{2}(\E)}^2
          +
          \Norm{\jump{\pd {tt} W'(u_h)}}^2_{\leb{2}(\E)}
        }
        \\
        &\qquad 
        +
        h 
        \qp{
          \frac{\mu}{\gamma^2}\Norm{\jump{\pd t \tau_h}}_{\leb{2}(\E)}^2
          +
          \frac{\mu}{\gamma^2}\Norm{\jump{\pd t W'(u_h)}}_{\leb{2}(\E)}^2
          +
          \Norm{\jump{\pd t v_h}}_{\leb{2}(\E)}^2
        }
        \\
        &\qquad 
        +
          \frac{1}{\gamma^2}\sum_{K\in\T{}} h_K^{2p+2} \bigg(
          \Norm{\pd t W'(u_h)}_{\sobh{p+1}(K)}^2
          +
        \Norm{W'(u_h)}_{\sobh{p+1}(K)}^2
          \\
          &\qquad \qquad \qquad \qquad \qquad 
          +
        \Norm{\pd {tt} W'(u_h)}_{\sobh{p+1}(K)}^2
        \bigg).
        \end{split}
    \end{equation}
\end{theorem}
\begin{proof}
  The result follows from applying the Gronwall inequality to Lemma
  \ref{lem:rre-bound-discrete} and using the bounds provided from
  Lemmata \ref{lem:control-of-E1}, \ref{lem:control-of-E2} and
  \ref{lem:control-of-E3}.
\end{proof}

\begin{lemma}[A posteriori control of the initial entropy error]
  \label{lem:apost-initial}
  Let the conditions of Lemma \ref{lem:rre-bound-discrete} hold, then
  \begin{equation}
    \begin{split}
      \eta_R(0) 
      &\lesssim
      \enorm{\qp{u - u_h}(\cdot, 0)}^2
      +
      \Norm{\qp{v - v_h}(\cdot, 0)}_{\leb{2}(S^1)}^2
      +
      \gamma \Eta_1[u_h^0, \frac{1}{\gamma}\qp{\tau_h^0 - W'(u_h^0)}]^2
      \\
      &\qquad +
      \frac{h}{\gamma^2}
      \qp{
        \Norm{\jump{\tau_h^0}}^2_{\leb{2}(\E)}
        +
        \Norm{\jump{u_h^0}}^2_{\leb{2}(\E)}
      }
      \\
      & =:  
      \mathfrak{E}_0[u_h^0, v_h^0, \tau_h^0]
    \end{split}
  \end{equation}
  with $u_h^0:= u_h(\cdot,0),\  v_h^0:=v_h(\cdot,0),\ \tau_h^0:=\tau_h(\cdot,0).$
\end{lemma}
\begin{proof}
  Recall that 
  \begin{equation}
    \eta_R(t) =   
    \frac{1}{2}\int_{S^1} \qp{v(\cdot,t) - \cR[v_h(\cdot,t)]}^2 
    +
    \gamma\qp{\pd x u(\cdot,t) - \pd x \cR_1[u_h(\cdot,t)]}^2 
    +
    \frac{\mu}{4} \int_0^t |v- \cR[v_h]|_{\sobh{1}(S^1)}^2,
  \end{equation}
  then in view of the triangle inequality we have that 
  \begin{equation}
    \label{eq:etar-bound}
    \begin{split}
      \eta_R(0) 
      &\lesssim
      \Norm{v(\cdot,0) - v_h(\cdot,0)}_{\leb{0}(S^1)}^2
      +
      \Norm{\cR[v_h](\cdot,0) - v_h(\cdot,0)}_{\leb{2}(S^1)}^2
      \\
      &\qquad +
      \gamma\qp{
      \enorm{u(\cdot,0) - u_h(\cdot,0)}^2
      +
      \enorm{\cR_1[u_h](\cdot,0) - u_h(\cdot,0)}^2}.
    \end{split}
  \end{equation}
  To estimate $\eta_R(0)$ we follow analagous arguments as in Lemmata
  \ref{lem-bd-on-Ruh} and \ref{lem-bd-on-Rvh} noting the defintion of the initial conditions of the scheme (\ref{eq:uh-init}) and (\ref{eq:vh-init}), taking the one sided
  limit as $t\to 0^+$, concluding the proof.
\end{proof}

\begin{theorem}[A posteriori control of the reduced relative entropy error]
  \label{the:apost-rre-error}
  Let the conditions of Lemma \ref{lem:rre-bound-discrete} hold and define
  \begin{equation}
    \label{eq:entropy-error}
    e_R(t) :=
    \sqrt{\gamma}
    \enorm{\qp{u - u_h}(\cdot,t)}
    +
    \Norm{\qp{v - v_h}(\cdot,t)}_{\leb{2}(S^1)}
    + 
    \sqrt{\frac{\mu}{4} \int_0^t \norm{(v - v_h )(\cdot,s)}_{\operatorname{dG}}\d s}
  \end{equation}
  then
  \begin{equation}
    \begin{split}
      e_R(t)
      &\lesssim
      \bigg[
      \mathfrak{E}_0[u_h^0, v_h^0, \tau_h^0]
      +
      \int_0^t 
      \mathfrak{E}_s[u_h, v_h, \tau_h]
      \d s
      \bigg]^{1/2}
      \exp\qp{\frac{1}{2}\int_0^tK[\cR_1[u_h]]\qp{s}\d s}
      \\
      &\qquad 
      +
      \sqrt{\gamma}
      \bigg[
      \Eta_1[u_h, \frac{1}{\gamma}\qp{\tau_h - W'(u_h)}]^2
      +
      \frac{h}{\gamma^2}
      \qp{
        \Norm{\jump{\tau_h}}^2_{\leb{2}(\E)}
        +
        \Norm{\jump{u_h}}^2_{\leb{2}(\E)}
      }
      \\
      &\qquad\qquad\qquad\qquad\qquad\qquad\qquad + 
      \frac{1}{\gamma^2}
      \sum_{K\in\T{}}
      h_K^{2p+2}
      \Norm{W'(u_h)}^2_{\sobh{p+1}(K)}
      \bigg]^{1/2}\\
      & \qquad 
      + \frac{\sqrt{\mu}}{2}\bigg[\int_0^t 
       \Eta_1[\pd t u_h, \frac{1}{\gamma}\qp{\pd t \tau_h - \pd t W'(u_h)}]^2
        + \frac{1}{\gamma^2} \sum_K h_K^{2p+2} \Norm{\partial_t W'(u_h)}_{\sobh{p+1}(K)}^2
      \\
      & \qquad \qquad
      + 
       \frac{h}{\gamma^2}\Norm{\jump{\partial_t W'(u_h)}}_{\leb{2}(\E)}^2
      +
      \frac{h}{\gamma^2}\Norm{ \jump{\partial_t \tau_h}}_{\leb{2}(\E)}^2
      +
      \Norm{\sqrt{h_\E^{-1}}\jump{v_h}}_{\leb{2}(\E)}^2
      \bigg]^{\frac{1}{2}}.
    \end{split}
  \end{equation}
\end{theorem}
\begin{proof}
  The proof follows by combining Theorem \ref{the:apost-rre}, Lemma
  \ref{lem:apost-initial} and noting that 
  \begin{equation}
    \begin{split}
      \gamma
      &\enorm{u(\cdot,t) - u_h(\cdot,t)}^2 +   \Norm{v(\cdot,t) - v_h(\cdot,t)}_{\leb{2}(S^1)}^2 + \frac{\mu}{4}\int_0^t \norm{v(\cdot,t) - v_h(\cdot,t)}_{\operatorname{dG}}^2
      \\
      &\qquad\qquad\lesssim
      \Norm{\cR[v_h](\cdot,t) - v_h(\cdot,t)}_{\leb{2}(S^1)}^2 + 
      \gamma\enorm{\cR_1[u_h](\cdot,t) - u_h(\cdot,t)}^2 \\
      &\qquad\qquad\quad +  \frac{\mu}{4}\int_0^t \norm{\cR[v_h](\cdot,t) - v_h(\cdot,t)}_{\operatorname{dG}}^2+ \eta_R(t),
    \end{split}
  \end{equation}
  concluding the proof.
\end{proof}

\begin{remark}[Optimality of the estimator]
  Given the a priori convergence result of Theorem \ref{the:apriori} we
  may infer that the indicator proposed in Theorem
  \ref{the:apost-rre-error} is of optimal order in the case of smooth initial data. Indeed, the leading
  order terms are given in the first three lines of
  (\ref{eq:estimator-et}). These are all $\Oh(h^{2p})$ in view of
  Assumption \ref{ass:a posteriori-control}, the boundedness of the
  solution, hence giving control of $W(u_h)$, and inverse inequalities. As such the full estimator in Theorem \ref{the:apost-rre-error} will be $\Oh(h^p)$. We refer the reader to \cite[Rem 3.6]{MN06} for a more detailed explanation of some of the terms.
\end{remark}

\begin{corollary}[A posteriori control of the modified relative entropy error]
  \label{eq:apost-mrre}
  Let the conditions of Lemma \ref{lem:rre-bound-discrete} hold and define
  \begin{multline}
    e_M(t) := 
    \sqrt{\gamma}
    \enorm{\qp{u - u_h}(\cdot,t)}
    +
    \Norm{\qp{u - u_h}(\cdot,t)}_{\leb{2}(S^1)}
    +
    \Norm{\qp{v - v_h}(\cdot,t)}_{\leb{2}(S^1)}
    \\
    + 
    \sqrt{\frac{\mu}{4} \int_0^t \norm{(v - v_h )(\cdot,s)}_{\operatorname{dG}}\d s}
  \end{multline}
  then
  \begin{equation}
    \begin{split}
      e_M(t)
      &\lesssim
      \bigg[
      \mathfrak{E}^M_0[u_h(0), v_h(0), \tau_h(0)]
      +
      \int_0^t 
      \mathfrak{E}_s[u_h(s), v_h(s), \tau_h(s)]
      \d s
      \bigg]^{1/2}
      \exp\qp{\frac{1}{2}\int_0^t \widetilde K[\cR_1[u_h]]\qp{s}\d s}
      \\
      &\qquad 
      +
      \sqrt{\gamma}
      \bigg[
      \Eta_1[u_h, \frac{1}{\gamma}\qp{\tau_h - W'(u_h)}]^2
      +
      \frac{h}{\gamma^2}
      \qp{
        \Norm{\jump{\tau_h}}^2_{\leb{2}(\E)}
        +
        \Norm{\jump{u_h}}^2_{\leb{2}(\E)}
      }
      \\
      &\qquad\qquad\qquad\qquad\qquad\qquad\qquad+ 
      \frac{1}{\gamma^2}
      \sum_{K\in\T{}}
      h_K^{2p+2}
      \Norm{W'(u_h)}^2_{\sobh{p+1}(K)}
      \bigg]^{1/2}\\
      & \qquad
      + \frac{\sqrt{\mu}}{2}\bigg[\int_0^t 
       \Eta_1[\pd t u_h, \frac{1}{\gamma}\qp{\pd t \tau_h - \pd t W'(u_h)}]^2
        + \frac{1}{\gamma^2} \sum_K h_K^{2p+2} \Norm{\partial_t W'(u_h)}_{\sobh{p+1}(K)}^2
      \\
      & \qquad \qquad
      + 
       \frac{h}{\gamma^2}\Norm{\jump{\partial_t W'(u_h)}}_{\leb{2}(\E)}^2
      +
      \frac{h}{\gamma^2}\Norm{ \jump{\partial_t \tau_h}}_{\leb{2}(\E)}^2
      +
      \Norm{\sqrt{h_\E^{-1}}\jump{v_h}}_{\leb{2}(\E)}^2
      \bigg]^{\frac{1}{2}}.
    \end{split}
  \end{equation}
\end{corollary}

\begin{remark}[Comparing the bounds of Theorem \ref{the:apost-rre-error} and
  Corollary \ref{eq:apost-mrre}]
  The main difference of the final a posteriori bounds in Theorem
  \ref{the:apost-rre-error} and Corollary \ref{eq:apost-mrre} is the
  exponential accumulation in time. In both results, the estimator is
  only valid for $\gamma > 0$, but in Corollary \ref{eq:apost-mrre} it does not depend exponentially on $\frac{1}{\gamma}$.
  The result for the reduced relative
  entropy (Theorem \ref{the:apost-rre-error}) is valid in the case $\mu =
  0$, however that of the modified relative entropy blows up
  as $\mu \to 0$ as $\widetilde{K} \sim \tfrac{1}{\mu}$. Note, however,
  that the $K[\cR_1[u_h]]$ contains the Lipschitz constant of
  $\cR_1[u_h]$ whereas $\widetilde K[\cR_1[u_h]]$ behaves like
  $\Norm{\cR_1[u_h]}_{\leb{\infty}}$.
\end{remark}

\section{Numerical experiments}
\label{sec:numerics}
In this section we conduct some numerical benchmarking on the estimator presented.

\begin{definition}[Estimated order of convergence]
  \label{def:EOC}
  Given two sequences $a(i)$ and $h(i)\downto0$,
  we define estimated order of convergence
  (\EOC) to be the local slope of the $\log a(i)$ vs. $\log h(i)$
  curve, i.e.,
  \begin{equation}
    \EOC(a,h;i):=\frac{ \log(a(i+1)/a(i)) }{ \log(h(i+1)/h(i)) }.
  \end{equation}
\end{definition}

\begin{definition}[Effectivity index]
  \label{def:EI}
  The main tool deciding the quality of an estimator is the
  effectivity index (\EI) which is the ratio of the error and the
  estimator, \ie
  \begin{equation}
    \EI:= 
    \frac{\max_{t} \mathfrak{H}_R}{\Norm{e_R}_{\leb{\infty}(0,T)}}.
  \end{equation}
\end{definition}

\begin{remark}[Computed indicator]
  In the numerical experiments we compute the indicator
  \begin{equation}
    \label{eq:ind-1}
    \begin{split}
      \mathfrak{H}_R &:= \qp{\int_0^t \widetilde{\mathfrak E}}^{1/2} 
      + 
      \sqrt{\gamma} \qp{\Eta_1[u_h, \frac{1}{\gamma}\qp{\tau_h - W'(u_h)}]^2
      +
      \frac{h}{\gamma^2}
      \qp{
        \Norm{\jump{\tau_h}}^2_{\leb{2}(\E)}
        +
        \Norm{\jump{u_h}}^2_{\leb{2}(\E)}
      }
      }
      \\
      &\qquad +
      \frac{\sqrt{\mu}}{2}\bigg[\int_0^t 
        \Eta_1[\pd t u_h, \frac{1}{\gamma}\qp{\pd t \tau_h - \pd t W'(u_h)}]^2
        + 
        \frac{h}{\gamma^2}\Norm{\jump{\partial_t W'(u_h)}}_{\leb{2}(\E)}^2
        +
        \Norm{\sqrt{h_\E^{-1}}\jump{v_h}}_{\leb{2}(\E)}^2
        \bigg]^{\frac{1}{2}},
    \end{split}
  \end{equation}
  where $\Eta_1$ is the elliptic estimator given in Example \ref{ex:ip-apost} and 
  \begin{equation}
    \label{eq:ind-2}
    \begin{split}
      \widetilde{\mathfrak{E}} 
      :=
      \Norm{\sqrt{h_\E^{-1}}\jump{u_h}}_{\leb{2}(\E)}^2
      +
      \mu\Norm{\sqrt{h_\E^{-1}}\jump{\pd t u_h}}_{\leb{2}(\E)}^2
      +
      \sum_{K\in\T{}} h_K^{2p} \Norm{W'(u_h)}_{\sobh{p+1}(K)}^2.
    \end{split}
  \end{equation}
  The terms in the analytic estimator given in Theorem
  \ref{the:apost-rre-error} which are not included in the computed
  indicator are of higher order, thus $\mathfrak{H}_R$ represents the
  dominant part of the analytic estimator.

  Notice also that we do not compute $K[\cR_1[u_h]]$. Of course we
  can, using the elliptic regularity of $\cR_1[u_h]$ we have that 
  \begin{equation}
    \Norm{\cR_1[u_h]}_{\sob{1}{\infty}(S^1)} 
    \lesssim 
    \Norm{\cR_1[u_h]}_{\sobh{2}(S^1)} 
    \lesssim
    \frac{1}{\gamma}
    \Norm{P^C_{p+1}[W'(u_h)] - D^+[\tau_h]}_{\leb{2}(S^1)}.
  \end{equation}
  As such, due to the regularity assumed on $u$ we have that
  $\cR_1[u_h]$ cannot blow up as $h\to 0$. Hence $\exp{\int_0^t K[\cR_1[u_h]]}$ must
  behave like a multiplicative constant.
\end{remark}

\subsection{Test 1: Benchmarking against known solution}

In this test we benchmark the numerical algorithm presented in
\S \ref{sec:discretisation} and the estimator given in Theorem
\ref{the:apost-rre} against a steady state solution of the regularised
elastodynamics system (\ref{eq:model-prob}) on the domain $\W =
[-1,1]$.

We take the double well 
\begin{equation}
  \label{eq:double-well}
  W(u) := \qp{u^2 - 1}^2,
\end{equation}
then a steady state solution to the regularised elastodynamics system
is given by
\begin{gather}
  u(x,t) 
  =
  \tanh\qp{ x \sqrt{\frac{2}{\gamma}}},
  \qquad v(x,t) \equiv 0 \Foreach t.
\end{gather}

The temporal derivatives in (\ref{eq:ind-1})--(\ref{eq:ind-2}) are approximated using difference quotients. We use the approximation
\begin{equation}
  \pd t u_h(t_n) \approx \frac{u_h^n - u_h^{n-1}}{\delta t},
\end{equation}
where $u_h^n$ denotes the fully discrete approximation at time $t_n$ and $\delta t$ the timestep.

For the implementation we are using natural boundary conditions, that is
\begin{equation}
  \partial_x u_h = v_h = 0 \text{ on } [0,T) \times \partial\W,
\end{equation}
rather than periodic.

Tables
\ref{table:p1-gamma-10-3}--\ref{table:p3-gamma-10-3} detail three
experiments aimed at testing the convergence properties for the scheme and estimator 
using piecewise discontinuous elements of various orders ($p=1$ in
Table \ref{table:p1-gamma-10-3}, $p=2$ in Table
\ref{table:p2-gamma-10-3} and $p=3$ in Table
\ref{table:p3-gamma-10-3}).

\begin{table}[h!]
  \caption{\label{table:p1-gamma-10-3} Test 1 :  In this test we benchmark a
    stationary solution of the regularised elastodynamics system using the discretisation
    (\ref{eq:discrete-scheme}) with piecewise linear elements ($p =
    1$). The temporal discretisation is a $2$nd order Crank--Nicolson method and we choose $\delta t = 1/N^2$ and $T=50$. We look at the reduced relative entropy error $e_R$ and the computed estimator $\mathfrak{H}_R$. In this test we choose
    $\gamma = \mu = 10^{-2}$. Notice that the estimator is robust, that is, it converges to zero at the same rate as the error.}
  \begin{center}
    \begin{tabular}{c|c|c|c|c|c}
      $N$ & $\Norm{e_R}_{\leb{\infty}(0,T)}$ & EOC & $\mathfrak{H}_R$ & EOC & EI \\
      \hline
      16 & 6.483984e+00 & 0.000 & 1.208189e+02 & 0.000 &  18.63 
      \\
      32 & 6.611057e+00 & -0.028 & 4.226465e+01& 1.515 & 6.39
      \\
      64 & 9.489125e-01  & 2.800 & 1.382063e+01 & 1.613 & 14.56
      \\
      128& 4.551683e-01 & 1.060 & 5.551150e+00  & 1.316 & 12.20
      \\
      256& 1.810851e-01 & 1.330 & 2.547156e+00 & 1.124 &  14.07
      \\
      512 & 9.046446e-02 & 1.001 & 1.225005e+00 & 1.056 & 13.54
      \\
      1024 & 4.513328e-02 & 1.003 & 6.124622e-01 & 1.000 & 13.57
    \end{tabular}
  \end{center}
\end{table}

\begin{table}[h!]
  \caption{\label{table:p2-gamma-10-3} Test 1 : The test is the same as in
    Table \ref{table:p1-gamma-10-3} with the exception that
    we take $p=2$. Notice that the estimator is robust, that is, it converges to zero at the same rate as the error.}
  \begin{center}
    \begin{tabular}{c|c|c|c|c|c}
      $N$ & $\Norm{e_R}_{\leb{\infty}(0,T)}$ & EOC & $\mathfrak{H}_R$ & EOC & EI \\
      \hline
      16 & 4.604586e+00 & 0.000 & 8.660246e+01 & 0.000 & 18.81
      \\
      32 & 5.714498e-01 & 3.010 & 2.239886e+01 & 1.951 & 39.20
      \\
      64 & 1.590671e-01 & 1.845 & 7.326793e+00 & 1.612 & 46.06
      \\
      128& 1.924736e-02 & 3.047 & 2.886101e+00 & 1.344 & 149.94
      \\
      256& 3.849299e-03 & 2.322 & 8.000554e-01 & 1.851 & 207.84
      \\
      512& 8.598334e-04 & 2.163 & 2.043625e-01 & 1.969 & 237.68
      \\
      1024&2.155231e-04 & 1.996 & 5.134256e-02 & 1.993 & 238.22
    \end{tabular}
  \end{center}
\end{table}

\begin{table}[h!]
  \caption{\label{table:p3-gamma-10-3} Test 1 : The test is the same as in
    Table \ref{table:p1-gamma-10-3} with the exception that
    we take $p=3$. Notice that the estimator is robust, that is, it converges to zero at the same rate as the error.}
  \begin{center}
    \begin{tabular}{c|c|c|c|c|c}
      $N$ & $\Norm{e_R}_{\leb{\infty}(0,T)}$ & EOC & $\mathfrak{H}_R$ & EOC & EI \\
      \hline
      16 & 4.618174e-01 & 0.000 & 3.840195e+01 & 0.000 & 83.15
      \\
      32 & 4.144471e-01 & 0.156 & 1.828968e+01 & 1.070 & 44.13
      \\
      64 & 7.399393e-02 & 2.486 & 6.327297e+00 & 1.531 & 85.51
      \\
      128& 1.036685e-02 & 2.835 & 6.845839e-01 & 3.208 & 66.04
      \\
      256& 1.291002e-03 & 3.005 & 6.165101e-02 & 3.473 & 47.75
      \\
      512& 1.602172e-04 & 3.010 & 6.905804e-03 & 3.158 & 43.10
      \\
      1024&2.010349e-05 & 2.994 & 8.576527e-04 & 3.009 & 42.66
    \end{tabular}
  \end{center}
\end{table}

\subsection{Test 2: Test problem with smooth initial data}

In this case we benchmark an unknown solution to
(\ref{eq:model-prob}). The initial conditions are smooth and taken to
be
\begin{gather}
  u(x,0) 
  =
  \frac{1}{100}\sin{50 \pi x},
  \qquad v(x,0) \equiv 0.
\end{gather}
The double well is again given by (\ref{eq:double-well}) and $\W =
[0,1]$. We summarise the results of this experiment in Table
\ref{tab:test2} and Figure \ref{fig:1d-smooth}.

\begin{table}[h!]
  \caption{\label{tab:test2} Test 2 : In this test we conduct a simulation with smooth initial conditions when the exact solution is unknown. The temporal discretisation is a $2$nd order Crank--Nicolson method and we choose $\delta t = 1/N^2$ and $T=50$. We look at the computed estimator $\mathfrak{H}_R$ and its convergence. In this test we choose
    $\gamma = 10^{-3}, \mu = 10^{-1}$. Note that the estimator converges at the same rates as were observed for the known solution in Test 1. Since the estimator is an upper bound for $e_R$ we have that $e_R$ is also converging optimally.}
  \begin{center}
    \begin{tabular}{c|c|c|c|c|c|c}
      $N$ & \multicolumn{2}{c|}{$p=1$} & \multicolumn{2}{c|}{$p=2$} & \multicolumn{2}{c}{$p=3$}
      \\
          & $\mathfrak{H}_R$ & EOC & $\mathfrak{H}_R$ & EOC & $\mathfrak{H}_R$ & EOC
      \\
      \hline
      16 & 1.188578e+04 & 0.000  & 5.777079e+03 & 0.000 & 3.246980e+03 & 0.000 
      \\
      32 & 4.669754e+03 & 1.348  & 4.358024e+03 & 0.407 & 5.753467e+03 & -0.825
      \\
      64 & 5.463170e+03 & -0.226 & 1.236485e+03 & 1.817 & 1.428526e+03 & 2.010 
      \\
      128 & 3.766053e+03 & 0.537 & 3.129763e+02 & 1.982 & 2.185178e+02 & 2.709 
      \\
      256 & 2.047099e+03 & 0.880 & 7.836314e+01 & 1.998 & 2.582443e+01 & 3.081 
      \\
      512 & 1.046615e+03 & 0.968 & 1.895438e+01 & 2.048 & 3.485471e+00 & 2.889 
      \\
      1024& 5.256529e+02 & 0.994 & 4.639561e+00 & 2.031 & 4.454392e-01 & 2.968 
    \end{tabular}
  \end{center}
\end{table}

\begin{figure}[h!]
  \caption[]
          {
            \label{fig:1d-smooth} 
            Test 2 : The solution, $u_h$ to
            the regularised elastodynamics system at various values of $t$. }
          \begin{center}
            \subfigure[][$t=0.04$]{
              \includegraphics[scale=\figscale, width=0.47\figwidth]
                              {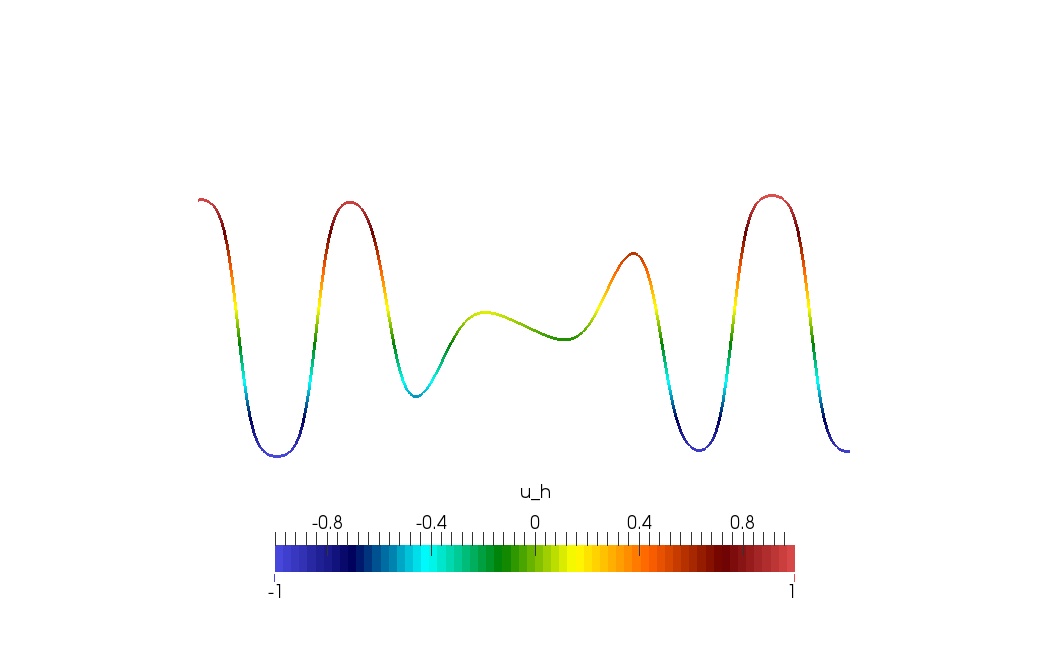}
            }
            \hfill
            \subfigure[][$t=0.17$]{
              \includegraphics[scale=\figscale, width=0.47\figwidth]
                              {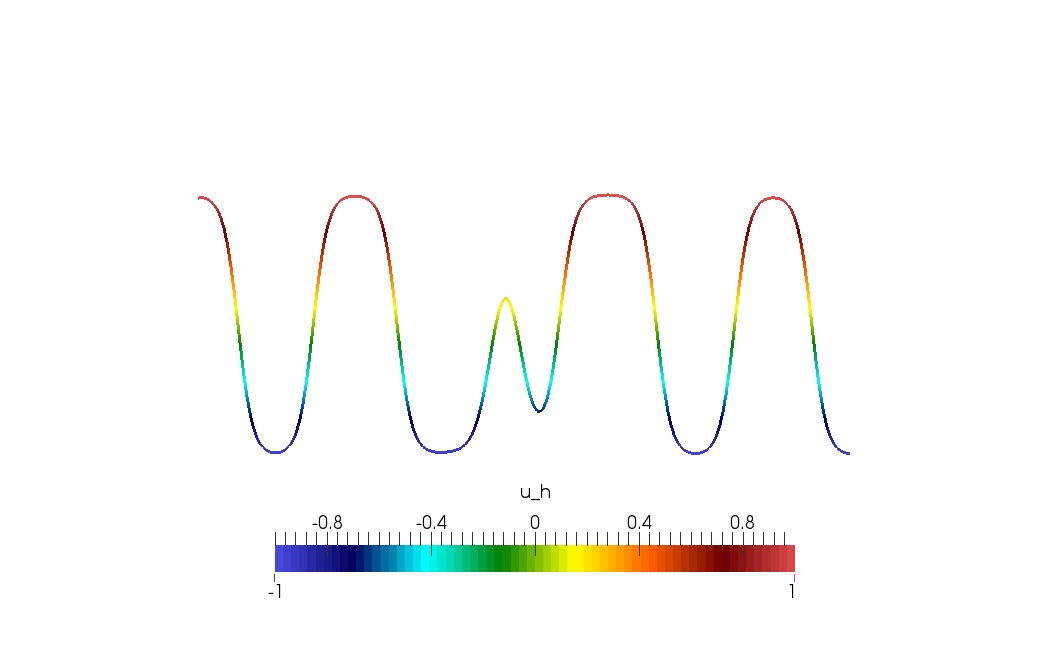}
                             
            }
            \hfill
            \subfigure[][$t=1.244$]{
              \includegraphics[scale=\figscale, width=0.47\figwidth]
              {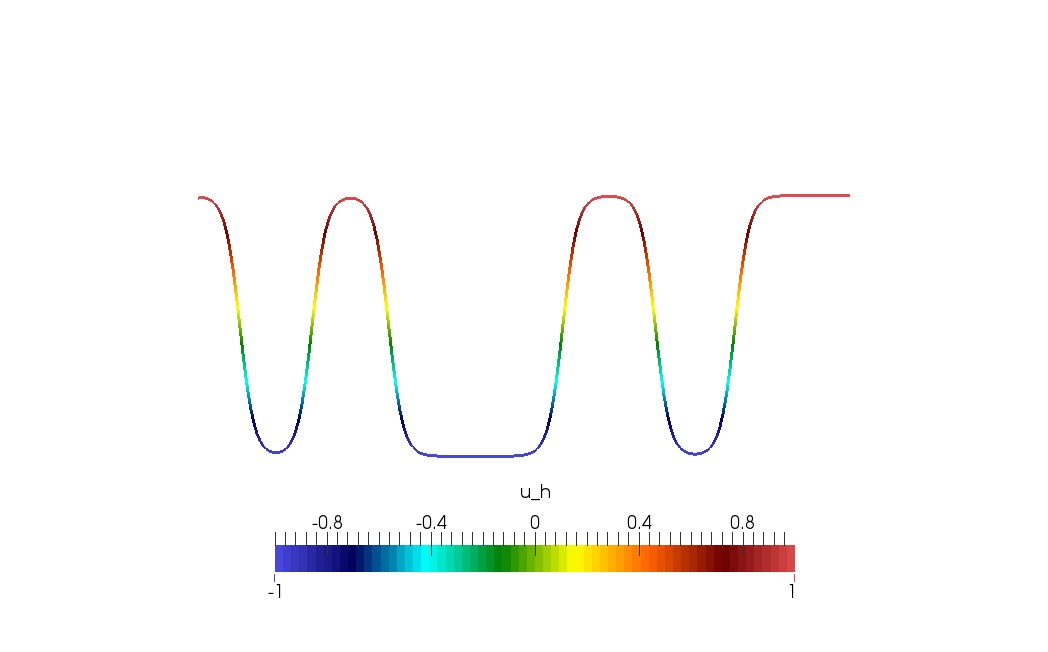}
                             
            }
            \hfill
            \subfigure[][$t=5$]{
              \includegraphics[scale=\figscale, width=0.47\figwidth]
              {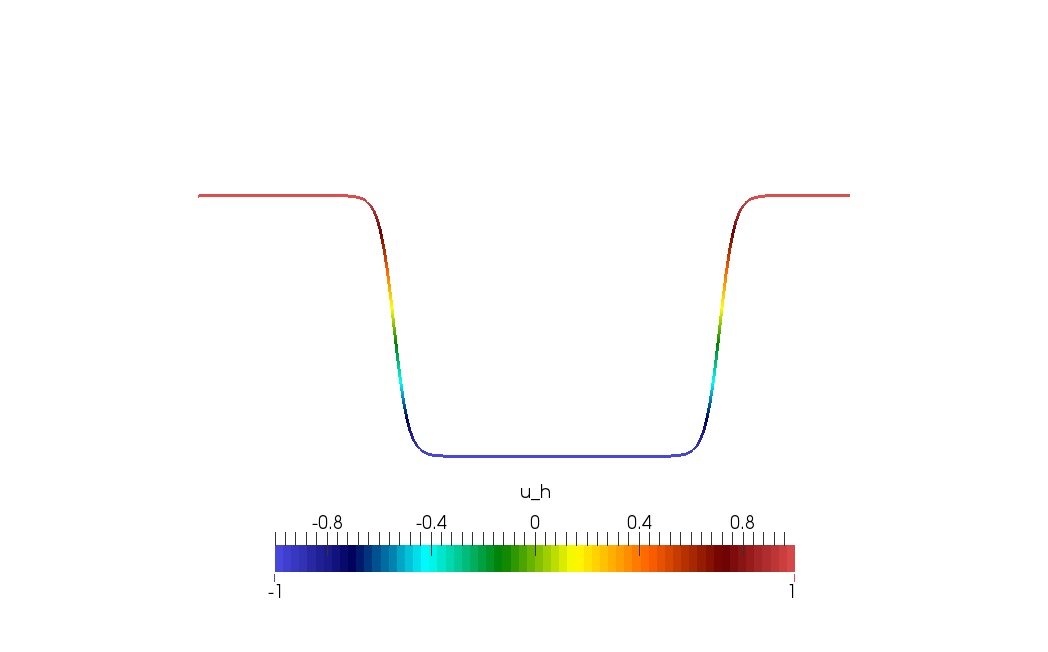}
                             
            }
          \end{center}
\end{figure}

\subsection{Test 3: Test problem with non smooth initial data}

In this case we benchmark an unknown solution to (\ref{eq:model-prob}). We take
\begin{gather}
  u(x,0) 
  =
  \begin{cases}
    \frac{1}{4}\qp{\cos{8\pi \norm{x - \frac{1}{2}}^2} + 1 } \text{ if } \norm{x - \frac{1}{2}}\leq \frac{1}{8}
    \\
    0 \text{ otherwise }
  \end{cases}
  ,
  \qquad v(x,0) \equiv 0.
\end{gather}
The initial conditions do not satisfy the assumptions of Theorem
\ref{the:existence}. In fact, $u_0 \in \sobh{2}(\W) \not \ \
\sobh{3}(\W)$. We summarise the results of this experiment in Table
\ref{tab:test3} and Figure \ref{fig:1d-notsmooth}.

\begin{table}[h!]
  \caption{\label{tab:test3} Test 3 : In this test we conduct a simulation with the initial conditions do not have the required regularity to yield a strong solution. The temporal discretisation is a $2$nd order Crank--Nicolson method and we choose $\delta t = 1/N^2$ and $T=50$. We look at the computed estimator $\mathfrak{H}_R$ and its convergence. In this test we choose
    $\gamma = 10^{-3}, \mu = 10^{-1}$. Note that the estimator converges at the same rates for all values of $p$. This is expected since the solution can not be $\leb{\infty}(0,T;\sobh{3}(\W))$ which would be a strong solution. Since the estimator is an upper bound for $e_R$ we have that $e_R$ is converging, however, suboptimally for $p>1$.}
  \begin{center}
    \begin{tabular}{c|c|c|c|c|c|c}
      $N$ & \multicolumn{2}{c|}{$p=1$} & \multicolumn{2}{c|}{$p=2$} & \multicolumn{2}{c}{$p=3$}
      \\
          & $\mathfrak{H}_R$ & EOC & $\mathfrak{H}_R$ & EOC & $\mathfrak{H}_R$ & EOC
      \\
      \hline
      16 & 1.188578e+04 & 0.000 & 3.246291e+03 & 0.000 & 2.911743e+03 & 0.000
      \\
      32 & 4.669754e+03 & 1.348 & 4.804005e+03 & -0.565& 4.669754e+03 & -0.682
      \\
      64 & 5.463170e+03 & -0.226& 4.729163e+03 & 0.023&  4.521050e+03 & 0.047
      \\
      128 & 3.766053e+03 & 0.537& 3.270226e+03 & 0.532 & 3.241540e+03 & 0.480
      \\
      256 & 2.047099e+03 & 0.880& 1.852629e+03 & 0.820 & 1.843336e+03 & 0.814
      \\
      512 & 1.046615e+03 & 0.968& 9.555120e+02 & 0.955 & 9.534166e+02 & 0.951
      \\
      1024& 5.256529e+02 & 0.994& 4.811709e+02 & 0.990 & 4.804343e+02 & 0.989
    \end{tabular}
  \end{center}
\end{table}

\begin{figure}[h!]
  \caption[]
          {
            \label{fig:1d-notsmooth} 
            Test 3 : The solution, $u_h$ to
            the regularised elastodynamics system at various values of $t$. }
          \begin{center}
            \subfigure[][$t=0.00$]{
              \includegraphics[scale=\figscale, width=0.47\figwidth]
                              {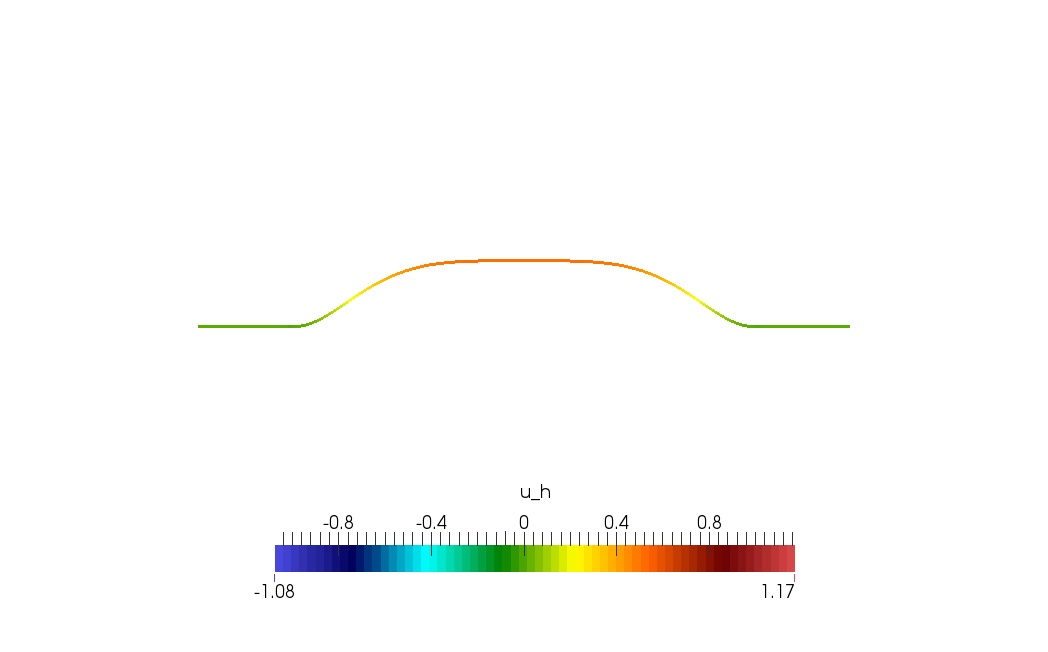}
            }
            \hfill
            \subfigure[][$t=0.1$]{
              \includegraphics[scale=\figscale, width=0.47\figwidth]
                              {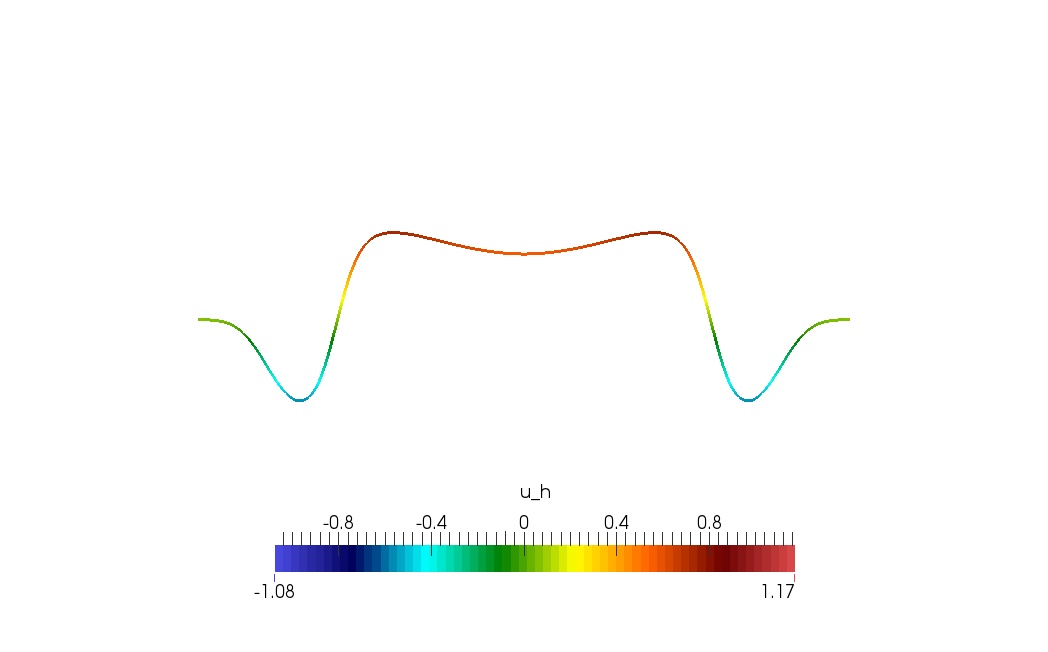}
                             
            }
            \hfill
            \subfigure[][$t=0.29$]{
              \includegraphics[scale=\figscale, width=0.47\figwidth]
              {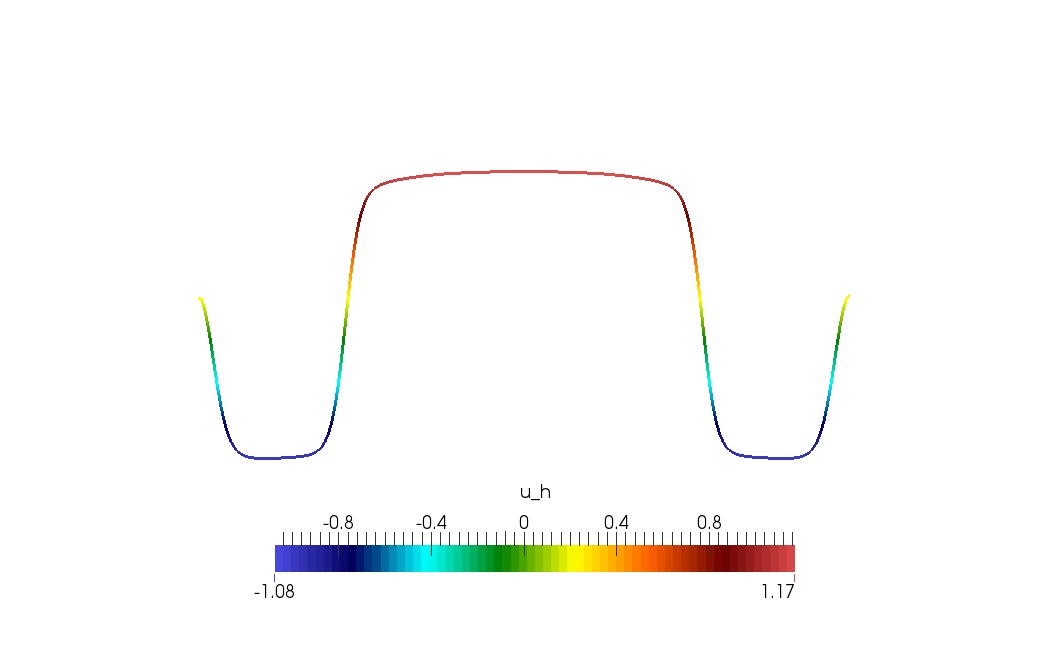}
                             
            }
            \hfill
            \subfigure[][$t=0.5$]{
              \includegraphics[scale=\figscale, width=0.47\figwidth]
              {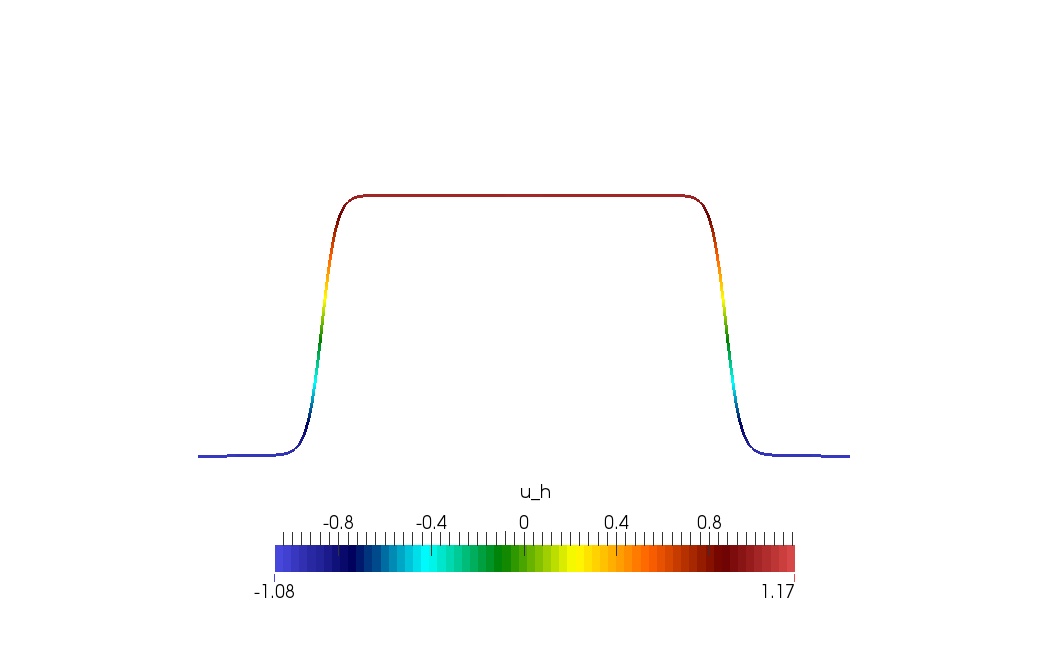}
                             
            }
          \end{center}
\end{figure}


\bibliographystyle{IMANUM-BIB}
\bibliography{nskbib,tristansbib,tristanswritings}

\end{document}